\documentclass{article}
\usepackage{geometry}
\usepackage[ansinew]{inputenc}
\usepackage{comment}
\usepackage{amssymb}
\usepackage{latexsym}
\usepackage{graphics}
\usepackage{epstopdf}
\usepackage{color}
\usepackage{fancybox}
\usepackage{amsmath,amsfonts}
 \usepackage[pagewise]{lineno}
 \usepackage{tikz}
\usepackage{graphicx}
\usepackage{tikz-3dplot}
\usepackage{setspace}
\usepackage{fancybox}
\usepackage{fancyhdr}
\usepackage{hyperref}
\begin{document}
\setlength{\parskip}{0.3\baselineskip}

\newtheorem{theorem}{Theorem}
\newtheorem{corollary}[theorem]{Corollary}
\newtheorem{lemma}[theorem]{Lemma}
\newtheorem{proposition}[theorem]{Proposition}
\newtheorem{definition}[theorem]{Definition}
\newtheorem{remark}[theorem]{Remark}
\renewcommand{\thefootnote}{\alph{footnote}}
\newenvironment{proof}{\smallskip \noindent{\bf Proof}: }{\hfill $\Box$\hspace{1in} \medskip \\ }


\newcommand{\beqaa}{\begin{eqnarray}}
\newcommand{\eeqaa}{\end{eqnarray}}
\newcommand{\beqae}{\begin{eqnarray*}}
\newcommand{\eeqae}{\end{eqnarray*}}


\newcommand{\sii}{\Leftrightarrow}
\newcommand{\imer}{\hookrightarrow}
\newcommand{\imerc}{\stackrel{c}{\hookrightarrow}}
\newcommand{\Con}{\longrightarrow}
\newcommand{\con}{\rightarrow}
\newcommand{\conf}{\rightharpoonup}
\newcommand{\confe}{\stackrel{*}{\rightharpoonup}}
\newcommand{\pbrack}[1]{\left( {#1} \right)}
\newcommand{\sbrack}[1]{\left[ {#1} \right]}
\newcommand{\key}[1]{\left\{ {#1} \right\}}
\newcommand{\dual}[2]{\langle{#1},{#2}\rangle}
\newcommand{\intO}[1]{\int_{\Omega}{#1}\, dx}

\newcommand{\R}{{\mathbb R}}
\newcommand{\N}{{\mathbb N}}
\newcommand{\Z}{{\mathbb Z}}

\newcommand{\cred}[1]{\textcolor{red}{#1}}

\title{\bf Regularidade do Sistema de Timoshenko com Termoelasticidade do Tipo III e Amortecimento Fracionário
\\
Regularity  of the Timoshenko's System with Thermoelasticity of Type III  and  Fractional Damping}
\author{Filomena Barbosa Rodrigues Mendes, \\
Department of Electrical Engineering,  Federal University of Technology of Paraná, Brazil, \\
Lesly Daiana B. Sobrado\\
Institute of Mathematics,  Federal University of Rio of Janeiro,  Brazil  and 
\\
Fredy Maglorio  S.  Suárez*\\
Department of Mathematics,  Federal University of Technology of  Paraná,  Brazil
}
\date{}
\maketitle
\let\thefootnote\relax\footnote{
{\it Email address:}  {\rm fredy@utfpr.edu.br:} Fredy M.  S.  Suárez $^*$Corresponding author. Published in Editora Atena, as a chapter of the book ``Ciências exatas e da terra: teorias e princípios 2",  on August 2, 2023: DOI: 10.22533/at.ed.3742302081.}
\begin{abstract}	
The article presents  the study of the regularity of two thermoelastic beam systems defined by the Timoshenko beam model coupled with the heat conduction of Green-Naghdiy theory of  type III,  both mathematical models are differentiated by their coupling terms that arise as a consequence of the constitutive laws initially considered.   The systems presented in this work  have 3 fractional dampings: $\mu_1(-\Delta)^\tau \phi_t$,  $\mu_2(-\Delta)^\sigma \psi_t$ and   $K(-\Delta)^\xi \theta_t$, where $\phi,\psi$ and $\theta$ are transverse displacement, rotation angle and empirical temperature of the bean  respectively and the parameters  $(\tau,\sigma,\xi)\in [0,1]^3$.  It is noted that for values 0 and 1 of the parameter $\tau$,  the so-called frictional or viscous damping will be faced, respectively.   The main  contribution of this article is to show that the corresponding semigroup $S_i(t)=e^{\mathcal{B}_it}$, with $i=1,2$,  is of Gevrey class 
$s>\frac{r+1}{2r}$  for   $r=\min \{\tau,\sigma,\xi\}$ for  all
 $(\tau,\sigma,\xi )\in R_{CG}:= (0, 1)^3$.  It is also showed  that $S_1(t)=e^{\mathcal{B}_1t}$ is analytic in the region $R_{A_1}:=\{(\tau,\sigma, \xi )\in [\frac{1}{2},1]^3\}$ and $S_2(t)=e^{\mathcal{B}_2t}$ is analytic in the region  $R_{A_2}:=\{(\tau,\sigma, \xi )\in [\frac{1}{2},1]^3/ \tau=\xi\}$.
\end{abstract}

\setcounter{equation}{0}

\section{Introduction}

In this article,  it is investigated the regularity of the semigroup associated with the thermoelastic beam system where the transversal vibrations are given by  Timo-shenko's model (See Timoshenko \cite{Timoshenko1921}) and the balance of the energy is described by the Green-Naghdi theory,  known as thermo-elasticity of type III(See Green and Naghdi \cite{GN1991}).  The equation of motion and energy balance for this thermoelastic Timoshenko system is given by
\begin{eqnarray*}
\rho  \mathcal{A} \phi_{tt}-S_x=0=0,\qquad \rho I\psi_{tt}-M_x+S=0,\qquad
\rho_3u_t+q_x+\gamma\psi_{xt} =  0, \quad  x\in (0,L), \;t\in \mathbb{R}^+.
\end{eqnarray*}
The constant $\rho$ denotes the density,   $  \mathcal{A}$ the cross-sectional area, and $I$ the area moment of inertia.  By $S$ one denotes the shear force,  $M$ is the bending moment and $q$ is the heat flux.  The function $\phi$ is the transverse displacement,  $\psi$  is the rotation angle of a filament of the beam and $u$ is the temperature difference.  Here,  $t$ is the time variable and $x$ is the coordinate space along the beam. 

In this investigation,  two systems that differ only in their coupling terms are analyzed,  which are obtained by starting from similar constitutive laws.

The first best-known mathematical model in the literature, and studied for example in \cite{LFJMRR2014}, \cite{Messaoudi2008} and \cite{Messaoudi2013}, starts from the following constitutive laws:
\begin{equation}\label{LeiC}
S= k   \mathcal{A}G(\phi_x+\psi),\qquad \qquad M=EI\psi_x+\beta u\qquad {\rm and}\qquad q=-\delta\alpha_x-K\alpha_{xt},
\end{equation}
 where $\alpha$ is the so-called thermal displacement whose time derivative is the empirical temperature $u$,  i.e.,  $\alpha_t=u, \; E$ and $G$ are elastic constants, $k$  the shear coefficient for measuring the stiffness of materials ($k < 1$), $\delta$ and $K$ denote the thermal conductivity,  $\beta$ the coefficient of linear thermal expansion and $\gamma$ a coupling constant.
 
To simplify the notation let us denote by $\rho_1 =\rho   \mathcal{A}$,  $ \rho_2= \rho I$,  $\kappa = k  \mathcal{A}G$ and $b = EI$.  
Under these conditions,  the system can be written as
\begin{eqnarray}
\label{Eq01}
\rho_1 \phi_{tt}-\kappa(\phi_x+\psi)_x&= & 0, \quad  x\in (0,L), \;t\in \mathbb{R}^+ ,\\
\label{Eq02}
\rho_2\psi_{tt}-b\psi_{xx}+\kappa(\phi_x+\psi)+\beta u_x&=& 0, \quad  x\in (0,L), \;t\in \mathbb{R}^+,\\
\label{Eq03}
\rho_3u_{tt}-\delta u_{xx}+\gamma\psi_{ttx}-Ku_{txx} &= & 0, \quad  x\in (0,L), \;t\in \mathbb{R}^+.
\end{eqnarray}
In order to exhibit the dissipative natural of system \eqref{Eq01}--\eqref{Eq03},  it is convenient to introduce a new variable (see \cite{ZhangZuazua2003}):
\begin{equation}
\label{Eq1.5Jaime}
\theta(t,x)=\int_0^tu(s,x)ds+\dfrac{1}{\delta}\chi(x),
\end{equation}
where $\chi\in H_0^1(0,L)$ solves the following Cauchy problem
\begin{equation}\label{Eq1.6Jaime}
\left\{\begin{array}{ccc}
\chi_{xx}=\rho_3 u_1-Ku_{0xx}+\gamma\psi_{1x} & {\rm in} & (0,L),\\
\chi(x)=0,  &x&=0, L.
\end{array}\right.
\end{equation}
Then,  considering 3 fractional damping   $\mu_1(-\Delta)^\tau \phi_t,\; \mu_2(-\Delta)^\sigma \psi_t$ and $K(-\Delta)^\xi \theta_t$  and using \eqref{Eq1.5Jaime} and \eqref{Eq1.6Jaime} the starting system \eqref{Eq01}--\eqref{Eq03} is transformed to
\begin{eqnarray}
\label{Eq01A}
\rho_1 \phi_{tt}-\kappa(\phi_x+\psi)_x+\mu_1(-\Delta)^\tau \phi_t&= & 0, \quad  x\in (0,L), \;t\in \mathbb{R}^+ ,\\
\label{Eq02A}
\rho_2\psi_{tt}-b\Delta\psi+\kappa(\phi_x+\psi)+\beta\theta_{tx}+\mu_2(-\Delta)^\sigma \psi_t&=& 0, \quad  x\in (0,L), \;t\in \mathbb{R}^+, \\
\label{Eq03A}
\rho_3\theta_{tt}-\delta\Delta\theta +\gamma\psi_{tx}+K(-\Delta)^\xi\theta_t &= & 0, \quad  x\in (0,L), \;t\in \mathbb{R}^+,
\end{eqnarray}
where the parameters $\tau$,  $\sigma$ and $\xi$ take values in the range $[0,1]$ and $\mu_1, \mu_2$ are positive real constants.

The second mathematical model addressed in this research that was proposed in 2017 by Santos et al. \cite{MSantosDilberto2017},   which consists of adjusting the couplings of the first system couplings to satisfy  type III heat conduction that acts mainly on the shear force, and as a result,  the transverse shear force $S =\kappa(\phi_x +\psi)$ is obtained from the transverse distribution of shear stresses acting on each cross-section of the beam.  For this purpose,  it is necessary to start with  the used constitutive laws 
\begin{equation}\label{LeiCNew}
S=\kappa(\phi_x+\psi)-\mu u,\qquad\quad M=b\psi_x\qquad\quad{\rm and}\qquad \quad q=-\delta \alpha_x-K\alpha_{xt},
\end{equation}
where $\alpha$ is the so-called thermal displacement, the time derivative of which is the empirical temperature $u$,  i.e.,  $\alpha_t = u$.  From these constitutive laws,  by making appropriate changes,   one  arrives at the system to follow (for more details see \cite{MSantosDilberto2017}):
\begin{eqnarray}
\label{Eq01B}
\rho_1 \phi_{tt}-\kappa(\phi_x+\psi)_x+\mu\theta_{tx}+\mu_1(-\Delta)^\tau \phi_t&= & 0, \quad  x\in (0,L), \;t\in \mathbb{R}^+ ,\\
\label{Eq02B}
\rho_2\psi_{tt}-b\Delta\psi+\kappa(\phi_x+\psi)-\mu\theta_t+\mu_2(-\Delta)^\sigma \psi_t&=& 0, \quad  x\in (0,L), \;t\in \mathbb{R}^+, \\
\label{Eq03B}
\rho_3\theta_{tt}-\delta \Delta\theta +\mu(\phi_x+\psi)_t+\gamma (-\Delta)^\xi\theta_t &= & 0, \quad  x\in (0,L), \;t\in \mathbb{R}^+.
\end{eqnarray}
To make the calculations in the results of the test more practical,   the two systems  are  abstractly rewritten by  using the operator: $A\colon D(A)\subset L^2(0,L)\con L^2(0,L)$, where
\beqaa\label{Omenoslaplaciano}
A=-\Delta=-(\,\cdot\,)_{xx},\quad D(A)=H^2(0,L)\cap H^1_0(0,L).
\eeqaa
It is known that the operator given in \eqref{Omenoslaplaciano} is self-adjoint, positive, and has inverse compact on a complex Hilbert space $D(A^0)=L^2(0,L)$. Therefore,  the operator $A^{\nu}$ is self-adjoint positive for all $\nu\in\mathbb{R}$,  bounded by $\nu\leq 0$,  and  the embedding
\begin{eqnarray*}
D(A^{\nu_1})\hookrightarrow D(A^{\nu_2}),
\end{eqnarray*}
is continuous for $\nu_1>\nu_2$. Here, the norm in $D(A^{\nu})$ is given by $\|u\|_{D(A^{\nu})}:=\|A^{\nu}u\|$, $u\in D(A^{\nu})$, where $\|\cdot\|$ denotes the norm in the Hilbert space $L^2(0,L)=D(A^0)$.  Some of these spaces are: $D(A^{1/2})=H_0^1(0,L)$, $D(A^0)=L^2(0,L)$ and $D(A^{-1/2})=H^{-1}(0,L)$. 

Studying the regularity (Gevrey class and/or Analyticity) of systems is relevant to science. In applied sciences, obtaining information about the solutions of the mathematical model under study such as regularity (smoothness) has the same importance as knowing the asymptotic behavior of the model's solution.  From a mathematical point of view,  when one talks about the regularity of solutions, concepts like differentiable,  Gevrey class,  and analyticity come to mind.   It is already known that semigroups $S(t)=e^{\mathcal{B}t}$ of the Gevrey class have better regular properties than a differentiable semigroup, but are less regular than an analytic semigroup.  The Gevrey rate $s>1$ 'measures' the degree of divergence of its power series. It should be noted that the Gevrey class or analyticity of the model in particular implies three very important properties: The first one is the property of the smoothing effect on initial data,  i. e.,  no matter how irregular the initial data is,  the solutions of the models become very smooth in positive finite-time. The second property is that systems are exponentially stable.  Finally, these systems benefit from the property of linear stability, which means that the type of the semigroup is equal to the spectral limit of its infinitesimal operator.

During the last decades, various authors have studied some physical phenomena for the Timoshenko beam system formulated in different mathematical models. Most of them focused on studying the asymptotic behavior,  always trying to obtain the best decay rate and using dissipations in some of the system equations.  In the following paragraphs, some of these investigations are mentioned.

In 1987, Kim and Renardy \cite{KR1987} studied the asymptotic behavior of the Timoshenko beam considering two boundary dissipations. They demonstrated the exponential decay of energy associated with the model using the multiplier technique and also determined numerical estimates of the eigenvalues of the operator associated with this system.  Later in 2008, Messaoudin and Said-Houari \cite{Messaoudi2008} also studied the asymptotic behavior of the Type III thermoelastic Timoshenko system with mixed boundary conditions (Dirichlet-Dirichlet-Neuman), they demonstrated using the energy method that if the velocities of the waves associated with the hyperbolic part of the system are equal, then the system decays exponentially. The complement of the study of the asymptotic behavior of this same system due to different mixed conditions (Neumann-Dirichlet-Neumann) was studied again in 2013 by Messaoudin and Fareh \cite{Messaoudi2013}, in this new research they show that if the speeds of the waves associated with the hyperbolic part of the system are different,  then the system decays polynomially.   In the following year, the work of Fatori et al. \cite{LFJMRR2014} who also studied this same system, but with two types of boundary conditions (Dirichlet-Dirichlet-Dirichlet) and (Dirichlet-Neumann-Dirichlet), using semigroup technique they showed that the corresponding semigroup is exponentially stable If and only if the velocities associated with the hyperbolic part of the system are equal, in the absence of exponential decay they show that the corresponding semigroup is polynomially stable and the determined rate is optimal.

More recently,  in 2017,  Santos et al.\cite{MSantosDilberto2017} studied the type III thermoelastic Timoshenko beam system given by:
\begin{eqnarray*}
\rho_1\phi_{tt}-\kappa(\phi_x+\psi)_x+\mu\theta_{tx} & = & 0, \qquad x\in (0,L),\; t\in\mathbb{R}^+,\\
\rho_2\psi_{tt}-b\psi_{xx}+\kappa(\phi_x+\psi)-\mu\theta_t &=& 0,\qquad x\in (0,L),\; t\in \mathbb{R}^+,\\
\rho_3\theta_{tt}-\delta\theta_{xx}+\mu(\phi_x+\psi)_t-\gamma\theta_{txx} &=& 0,\qquad x\in (0,L),\; t\in \mathbf{R}^+,
\end{eqnarray*}
and the boundary conditions (Dirichlet-Dirichlet-Dirichlet) given by
\begin{equation*}
\phi(0,t)=\phi(L,0)=\psi(0,t)=\psi(L,0)=\theta(0,t)=\theta(L,t)=0, \qquad\forall t>0,
\end{equation*}
or with the boundary conditions(Dirichlet-Neumann-Neumann)  given by
\begin{equation*}
\phi(0,t)=\phi(L,0)=\psi_x(0,t)=\psi_x(L,0)=\theta_x(0,t)=\theta_x(L,t)=0, \qquad\forall t>0,
\end{equation*}
note that this system differs slightly from the terms of coupling.  In this research, the authors also applied the semigroup technique and showed that the associated semigroup can be exponentially or polynomially stable depending on the relationships between the wave propagation velocity coefficients, specifically in the case of polynomial decay it was proved that the rate found is optimal.

About investigations the regularity of semigroups associated with various mathematical models, one could mention the paper/article of Fatori et al.\cite{LMJAIME2012},   whose work the authors studied the differentiability and analyticity,  in addition to studying the asymptotic behavior via semigroups.  For analyticity,  they used a theorem that can be found in the Liu-Zheng book\cite{LiuZ} or the work of Hao et al.\cite{Hao-2015}.  Other more recently published works have explored the regularity of solutions using the Gevrey class introduced in Taylor's thesis \cite{TaylorM}(1989). In the same direction, one could mention \cite{KAFSTebou2021, GRiveraO2022, GAJMRLiu2021}.

Among recent research that has studied the asymptotic behavior and/or regularity of models with fractional damping, one could mention the works of Sare et al.\cite{HSLiuRacke2019}, in that paper, the authors investigated coupled systems thermoelastic type,  where they address two cases with Fourier's heat law and the other with Cattaneo considering in both cases the rotational inertial term.  Furthermore,  they,   study the exponential stability,  possible regions of loss of exponential stability and polynomial stability,  and,  more recently, the work of Keyantuo et al.\cite{Tebou-2020}(2020) to be published.  In this latter  work,  the authors studied the thermoelastic plate model with a fractional Laplacian between the Euler-Bernoulli and Kirchhoff model with two types of boundary conditions.  In addition to studying the asymptotic and analytical behavior,  the authors show that the underlying semigroups are of Gevrey class $s$ for every $s>\frac{2-\theta}{2-4\theta}$ for both the clamped and hinged boundary conditions when the parameter $\theta$ lies in the interval $(0,1/2)$.  Moreover, one could cite the investigations \cite{RCRaposo2022,KZLZHugo2021,ZKLiuTebou2022,HPFredy2019,BrunaJMR2022, Tebou-2021}. 

One motivation for deciding to study the regularity (determination of the Gevrey classes and analyticity) of Timoshenko's systems was in the direction of complementing the work published in 2005 by Raposo et al. \cite{RFMN2005}, in this work the authors study the asymptotic behavior (exponential decay of the Timoshenko system), the studied system considers two weak (frictional) dampings given by the speed of the transversal and rotational movements. The investigated model is given by
\begin{eqnarray*}
\rho_1u_{tt}-k(u_x-\psi)_x+u_t=0\quad {\rm in}\quad (0,L)\times (0,\infty),\\
\rho_2\psi_{tt}-b\psi_{xx}+k(u_x-\psi)+\psi_t=0,\quad{\rm in}\quad (0,L)\times (0,\infty),\\
u(0,t)=u(L,t)=\psi(0,t)=\psi(L,t)=0\qquad t>0.
\end{eqnarray*}
The authors use the frequency domain technique (spectral characterization) to study the asymptotic behavior via semigroups.

This article is organized as follows. In section 2, using the theory of semigroup,   the well-posedness of both systems is studied.  Section 3 is dedicated to the study of the regularity of both models,  this subsection is subdivided into two parts: in subsection 3.1  Gevrey class $s_1$ of the semigroup associated with the first system $S_1(t)=e^{\mathcal{ B}_1t }$ in the region $R_{GC}=(0,1)^3$ is determined and it is shown that  $S_1(t)=e^{\mathcal{B}_1t}$ is analytic in the region $R_{A1}$.  Finally,  in subsection 3.2 we show that the Gevrey class $s_2$ of $S_2(t)=e^{\mathcal{B}_2t}$. The classes $s_i>\frac{r+1}{2r}$ where $r=\min\{\tau,\sigma,\xi\}$ for all  the parameters $(\tau,\sigma, \xi ) \in R_{GC}=(0,1)^3$ and $i=1,2$.  And this subsection is finished by proving that $S_2(t)=e^{\mathcal{B}_2t}$ is analytic when the 3 parameters $\tau$, $\sigma$,  and $\xi$ take values in the closed interval $[\frac{1}{2},1]$ such that $\tau=\xi$.  The
investigation ends with an observation regarding the asymptotic behavior of $S(t)=e^{\mathcal{B}_it}$ in which it is illustrated that the necessary estimates given in Lemmas \ref{LemaExponencial}, \ref{EixoImaginary01} and \ref{LemaExponencial2} and remark \eqref{OBS} imply that the semigroup associated with the system \eqref{Eq05A}--\eqref{Eq07A} and \eqref{Eq08A}--\eqref{Eq09A} and the system \eqref{Eq01B2}-\eqref{Eq09A}  respectivamently are exponentially stable for $(\tau, \sigma,\xi)\in [0, 1]^3$.
\section{Well-posedness: semigroup approach}
Using the $A$ operator defined in \eqref{Omenoslaplaciano} the system \eqref{Eq01A}--\eqref{Eq03A} will be given by
\begin{eqnarray}
\label{Eq05A}
\rho_1 \phi_{tt}-\kappa(\phi_x+\psi)_x+\mu_1A^\tau \phi_t&= & 0, \quad  x\in (0,L), \;t\in \mathbb{R}^+ ,\\
\label{Eq06A}
\rho_2\psi_{tt}+bA\psi+\kappa(\phi_x+\psi)+\beta\theta_{tx}+\mu_2A^\sigma \psi_t&=& 0, \quad  x\in (0,L), \;t\in \mathbb{R}^+, \\
\label{Eq07A}
\rho_3\theta_{tt}+\delta A\theta +\gamma\psi_{tx}+KA^\xi\theta_t &= & 0, \quad  x\in (0,L), \;t\in \mathbb{R}^+.
\end{eqnarray}
And again using the $A$ operator on \eqref{Eq01B}--\eqref{Eq03B},   leads to
\begin{eqnarray}
\label{Eq01B2}
\rho_1 \phi_{tt}-\kappa(\phi_x+\psi)_x+\mu\theta_{tx}+\mu_1A^\tau \phi_t&= & 0, \quad  x\in (0,L), \;t\in \mathbb{R}^+ ,\\
\label{Eq02B2}
\rho_2\psi_{tt}+bA\psi+\kappa(\phi_x+\psi)-\mu\theta_t+\mu_2A^\sigma \psi_t&=& 0, \quad  x\in (0,L), \;t\in \mathbb{R}^+, \\
\label{Eq03B2}
\rho_3\theta_{tt}+\delta A\theta +\mu(\phi_x+\psi)_t+\gamma A^\xi\theta_t &= & 0, \quad  x\in (0,L), \;t\in \mathbb{R}^+.
\end{eqnarray}
Additionally, the appropriate initial conditions for both systems are considered
\begin{equation}\label{Eq08A}
\hspace*{-0.2cm}\varphi(0,\cdot)=\phi_0, \quad \varphi_t(0,\cdot)=\phi_1,\quad \psi(0,\cdot)=\psi_0,\quad
 \psi_t(0,\cdot)=\psi_1,  \quad \theta(0,\cdot)=\theta_0\quad {\rm and}\quad \theta_t(0,\cdot)=\theta_1,
\end{equation}
and boundary conditions for both systems of type Dirichlet--Dirichlet--Dirichlet
\begin{equation}\label{Eq09A}
\hspace*{-0.5cm}\phi(t,0)=\phi(t,L)=0, \quad \psi(t,0)=\psi(t,L)=0, \quad\theta(t,0)=\theta(L,0)=0\quad  t>0.
\end{equation}
Taking $\Phi=\phi_t$, $\Psi=\psi_t$, $\Theta=\theta_t$,  considering $U=(\phi,\Phi,\psi, \Psi, \theta, \Theta)$ and $U_0=(\phi_0,\phi_1,\psi_0, \psi_1,\theta_0,\theta_1)$,  both  systems,  can be written in the following abstract framework
\begin{equation}\label{Fabstrata}
    \frac{d}{dt}U(t)=\mathcal{B}_i U(t),\quad    U(0)=U_0,
\end{equation}
where for the operator $\mathcal{B}_1$ for the first system is defined by
  \begin{gather} \label{operadorAgamma}
 \mathcal{B}_1U:=\left[   
 \begin{array}{c}
 \Phi\\
 \frac{1}{\rho_1}[\kappa(\phi_x+\psi)_x- \mu_1A^\tau \Phi]\\
 \Psi\\
- \frac{1}{\rho_2}[bA\psi+\kappa(\phi_x+\psi)+\beta\Theta_x +\mu_2 A^\sigma \Psi]\\
 \Theta\\
- \frac{1}{\rho_3}[\delta A\theta+\gamma\Psi_x+K A^\xi\Theta]
  \end{array}
  \right ],
  \end{gather}
and the operator $\mathcal{B}_2$ for the second system is defined by
 \begin{gather} \label{operadorBgamma}
 \mathcal{B}_2U:=\left[   
 \begin{array}{c}
 \Phi\\
 \frac{1}{\rho_1}[\kappa(\phi_x+\psi)_x-\mu\Theta_x- \mu_1A^\tau \Phi]\\
 \Psi\\
 \frac{1}{\rho_2}[-bA\psi-\kappa(\phi_x+\psi)+\mu\Theta -\mu_2 A^\sigma \Psi]\\
 \Theta\\
 \frac{1}{\rho_3}[-\delta A\theta-\mu(\Phi_x+\Psi)-\gamma A^\xi\Theta]
  \end{array}
  \right ],
  \end{gather}
  
for $U=(\phi,\Phi,\psi, \Psi, \theta, \Theta)$.  Both operators are defined in the energy space
$$              \begin{array}{ll}
               \mathcal{H}_1=\mathcal{H}_2=\mathcal{H}:= [D(A^\frac{1}{2})\times  D(A^0)]^3.
              \end{array}
$$
$\mathcal{H}_1$ is a Hilbert space with the inner product given by
\begin{eqnarray*}
\langle U, U^* \rangle_ {\mathcal{H}_1}  &:= &\rho_1 \beta\gamma\dual{\Phi}{\Phi^*}+\rho_2\kappa\gamma \dual{\Psi}{\Psi^*}+\beta \kappa\gamma\dual{\phi_x+\psi} {\phi_x^*+\psi^*}+b\kappa\gamma\dual{\psi_x} {\psi_x^*}\\
& &  +\beta\delta\kappa\dual{\theta_x}{\theta^*_x}+\beta\kappa K\dual{\Theta}{\Theta^*},
\end{eqnarray*}
for $U=(\phi,\Phi,\psi, \Psi, \theta, \Theta), U^*=(\phi^*,\Phi^*,\psi^* , \Psi^*, \theta^*, \Theta^*)\in \mathcal{H}$, 
and induced norm
\begin{eqnarray*}
\|U\|^2_ {\mathcal{H}_1}&:= &\rho_1\beta\gamma \|\Phi\|^2+\rho_2\kappa \gamma\|\Psi\|^2+ \beta\kappa\gamma\|\phi_x+\psi\|^2+b\kappa\gamma\|A^\frac{1}{2}\psi\|^2 +\beta\delta\kappa\|A^\frac{1}{2}\theta\|^2+\beta\kappa K\|\Theta\|^2.
\end{eqnarray*}
 In these conditions,  he domain of $\mathcal{B}_1$ is defined as
\begin{eqnarray}
\label{dominioB}
 \hspace*{-0.6cm}   \mathcal{D}(\mathcal{B}_1)   
  \hspace*{-0.3cm}   &:=& \hspace*{-0.3cm}  D(A)\times D(A^\frac{1}{2})\cap D(A^\tau)\times D(A) \times D(A^\frac{1}{2})\cap D(A^\sigma)\times D(A)\times D(A^\frac{1}{2})\cap D(A^\xi).
    \end{eqnarray}
    And $\mathcal{H}_2$  is a Hilbert space with the inner product given by
    \begin{eqnarray*}
\langle U, U^* \rangle_ {\mathcal{H}_2}  &:= &\rho_1\dual{\Phi}{\Phi^*}+\rho_2 \dual{\Psi}{\Psi^*}+ \kappa\dual{\phi_x+\psi} {\phi_x^*+\psi^*}+b\dual{\psi_x} {\psi_x^*}  +\delta\dual{\theta_x}{\theta^*_x}+\rho_3\dual{\Theta}{\Theta^*},
\end{eqnarray*}
for $U=(\phi,\Phi,\psi, \Psi, \theta, \Theta), U^*=(\phi^*,\Phi^*,\psi^* , \Psi^*, \theta^*, \Theta^*)\in \mathcal{H}$, 
and induced norm
\begin{eqnarray*}
\|U\|^2_ {\mathcal{H}_2}&:= &\rho_1 \|\Phi\|^2+\rho_2\|\Psi\|^2+ \kappa\|\phi_x+\psi\|^2+b\|A^\frac{1}{2}\psi\|^2 +\delta\|A^\frac{1}{2}\theta\|^2+\rho_3\|\Theta\|^2.
\end{eqnarray*}
 In these conditions,  the domain of $\mathcal{B}_2$ is defined as
\begin{eqnarray} \label{dominioB1}
\hspace*{-0.6cm}    \mathcal{D}(\mathcal{B}_2)
   &:=& D(A)\times D(A^\frac{1}{2})\cap D(A^\tau)\times D(A) \times D(A^\frac{1}{2})\cap D(A^\sigma)\times D(A)\times D(A^\frac{1}{2})\cap D(A^\xi).
    \end{eqnarray}
To show that both operators $\mathcal{B}_i$,   $i=1,2$  are generators of  a  $C_0-$\! semigroup we invoke a result from Liu-Zheng' book\cite{LiuZ}.
\begin{theorem}[see Theorem 1.2.4 in \cite{LiuZ}] \label{TLiuZ}
Let $ \mathcal{B}_i$ be a linear operator with domain $\mathcal{D}(\mathcal{B}_i)$ dense in a Hilbert space $\mathcal{H}$. If $ \mathcal{B}_i$ is dissipative and $0\in\rho( \mathcal{B}_i)$, the resolvent set of $ \mathcal{B}_i$, then $ \mathcal{B}_i$ is the generator of a $C_0$- semigroup of contractions on $\mathcal{H}$.

\end{theorem}
Next, we apply the Theorem \ref{TLiuZ} only for the first system, the proof of the second system is completely similar. Let us see that the operator $ \mathcal {B} _1$  defined in \eqref{operadorAgamma} satisfies the conditions of this theorem. Clearly,   $\mathcal{D}(\mathcal{B}_1)$ is dense in $\mathcal{H}$.  Effecting the internal product of $\mathcal{B}_1U$ with $U,$ lead to
\begin{equation}\label{eqdissipative}
\text{Re}\dual{\mathcal{B}_1U}{U}=  -\beta\gamma\mu_1\|A^\frac{\tau}{2}\Phi\|^2-\kappa\gamma\mu_2\|A^\frac{\sigma}{2} \Psi\|^2-\dfrac{\beta\kappa K^2}{\rho_3}\|A^\frac{\xi}{2}\Theta\|^2, \quad\forall\ U\in \mathcal{D}(\mathcal{B}_1),
\end{equation}
that is, the operator $\mathcal{B}_1$ is dissipative.

 To complete the conditions of the above theorem, it remains to show that $0\in\rho(\mathcal{B}_1)$. Let $F=(f^1,f^2,f^3,f^4,f^5, f^6)\in \mathcal{H}$, let us see that the stationary problem $ \mathcal{B}_1U=F$ has a solution $U=(\phi,\Phi, \psi, \Psi,\theta,\Theta)$.  From the definition of the operator  $\mathcal{B}_1$ given in
\eqref{operadorAgamma},  this system can be written as                     
\begin{align}
	\Phi=f^1,\qquad& \quad\quad  \kappa (\phi_x+\psi)_x=\rho_1f^2+\mu_1 A^\tau f^1, \label{exist-10A}\\
	\Psi=f^3,\qquad &  \quad\quad  bA\psi+\kappa (\phi_x+\psi) =-\rho_2f^4-\mu_2A^\sigma f^3-\beta f^5_x, \label{exist-20A}\\
\Theta=f^5, \quad \quad &  \quad \delta A\theta=-\gamma f^3_x-KA^\xi f^5-\rho_3f^6.
\label{exist-30A}	
\end{align}
Therefore, it is not difficult to see that there exists only one solution $\phi, \psi$ and $\theta$
of the system
\begin{multline}\label{Eliptico001}
 \kappa (\phi_x+\psi)_x=\rho_1f^2+\mu_1A^\tau f^1\qquad \in \qquad D(A^0),\\
  bA\psi+\kappa (\phi_x+\psi) =-\rho_2f^4-\mu_2A^\sigma f^3-\beta f^5_x\qquad \in \quad D(A^0),  \\
   \delta A\theta=-\gamma f^3_x-KA^\xi f^5-\rho_3 f^6 \qquad \in \quad D(A^0).
\end{multline}
Therefore: $\|U\|_\mathcal{H}\leq C\|F\|_\mathcal{H}$, 
which in particular implies that $\|\mathcal{B}_1^{-1}F\|_\mathcal{H}\leq C \|F\|_\mathcal{H}$, so we have  that   $0$ belongs to the resolvent set $\rho(\mathcal{B}_1)$.  Consequently, from Theorem \ref{TLiuZ}  we have  $\mathcal{B}_1$ as the generator of a contractions semigroup.

As a consequence of the above Theorem\eqref{TLiuZ},  it follows
\begin{theorem}
Given $U_0\in\mathcal{H}$ there exists a unique weak solution $U$ to  the problem \eqref{Fabstrata} satisfying 
$$U\in C([0, +\infty), \mathcal{H}).$$
Futhermore, if $U_0\in  D(\mathcal{B}_i^k), \; k\in\mathbb{N}$, then the solution $U$ of \eqref{Fabstrata} satisfies
$$U\in \bigcap_{j=0}^kC^{k-j}([0,+\infty),  D(\mathcal{B}_i^j).$$
\end{theorem}
\section{Regularization Results}
This section is divided into two subsections,   in each of them,  the Gevrey class and analytics of the corresponding systems are studied.  Using the characterization results presented in \cite{Tebou-2020} (adapted from \cite{TaylorM}, Theorem 4, p. 153])  it is shown that the corresponding semigroups $S_i(t) = e^{\mathcal{B}_it}$, $i=1,2$ are of Gevrey class  $s>\frac{r+1}{2r}$  for   $r=\min \{\tau,\sigma,\xi\}$ for all
 $(\tau,\sigma,\xi )\in R_{CG}:= (0, 1)^3$.  And for the study of analyticity the main tool used is the characterization of analytical semigroups due to Liu and Zheng\cite{LiuZ} (See the book by Liu-Zheng - Theorem 1.3.3).   It has proved  that $S_1(t)=e^{\mathcal{B}_1t}$ is analytic in the region $R_{A_1}:=[1/2,1]^3$ and $S_2(t)=e^{\mathcal{B}_2t}$ is analytic in the region $R_{A_2}:=\{(\tau,\sigma,\xi)\in [1/2,1]^3/ \tau=\xi\}$.

In what follows: $C$,  $C_\delta$ and $C_\varepsilon$ will denote positive constants that assume different values in different places.
\subsection{Regularity of the first system}

Next, two lemma are presented where two estimates are tested which are fundamental for the determination of  Gevrey class and the analytics of the associated semigroup $S(t)=e^{\mathcal{B}_1t}$.
\begin{lemma}\label{LemaExponencial}
Let $S(t)=e^{\mathcal{B}_1t}$  be a $C_0$-semigroup on contractions on Hilbert space                   $\mathcal{H}=\mathcal{H}_1$,  the solutions of the system \eqref{Eq05A}--\eqref{Eq07A} and \eqref{Eq08A}--\eqref{Eq09A}  satisfy the inequality
\begin{equation}\label{EstimaExponential} 
\limsup_{|\lambda|\to\infty}\|(i\lambda I-\mathcal{B}_1) ^{-1}\|_{\mathcal{L}(\mathcal{H})}<\infty.
\end{equation}
\end{lemma}
\begin{proof} 
To show the \eqref{EstimaExponential} inequality,  it suffices to show that,  given $\delta>0$ there exists a constant $C_\delta>0$ such that  the solutions of the system  \eqref{Eq05A}--\eqref{Eq07A} and \eqref{Eq08A}--\eqref{Eq09A}   for $|\lambda|>\delta$ satisfy the inequality
\begin{equation}\label{EstimaEquivExp1}
 \dfrac{\|U\|_{\mathcal{H}}}{\|F\|_{\mathcal{H}}}\leq C_\delta\Longleftrightarrow \|U\|_{\mathcal{H}}=\|(i\lambda I-\mathcal{B}_1)^{-1} F\|_{\mathcal{H}} \leq C_\delta \|F\|_\mathcal{H}.
\end{equation}
If $\lambda\in \mathbb{R}$ and $F=(f^1,f^2,f^3,f^4, f^5, f^6)\in \mathcal{H}$ then the solution $U=(\phi,\Phi,\psi, \Psi,\theta,\Theta)\in\hbox{D}(\mathcal{B}_1)$ the resolvent equation $(i\lambda I- \mathcal{B}_1)U=F$ can be written in the form
\begin{eqnarray}
i\lambda \phi-\Phi &=& f^1\quad\rm{in}\quad D(A^\frac{1}{2}),\label{esp-10}\\
i\lambda \Phi-\frac{\kappa}{\rho_1}(\phi_x+\psi)_x+\frac{\mu_1}{\rho_1}A^\tau \Phi&=& f^2\quad\rm{in}\quad D(A^0),\label{esp-20}\\
i\lambda  \psi-\Psi &=&f^3\quad \rm{in}\quad D(A^\frac{1}{2}),\label{esp-30}\\
i\lambda \Psi+\dfrac{b}{\rho_2}A\psi+\dfrac{\kappa}{\rho_2}(\phi_x+\psi)+\frac{\beta}{\rho_2}\Theta_x+\dfrac{\mu_2}{\rho_2}A^\sigma \Psi &=& f^4 \quad\rm{in}\quad D(A^0),\label{esp-40}\\
i\lambda\theta-\Theta&=&f^5\quad{ \rm in}\quad D(A^\frac{1}{2}),\label{esp-50}\\
i\lambda\Theta+\dfrac{\delta}{\rho_3}A\theta+\dfrac{\gamma}{\rho_3}\Psi_x+\dfrac{K}{\rho_3}A^\xi\Theta&=& f^6\quad {\rm in}\quad D(A^0). \label{esp-60}
\end{eqnarray}
Using the fact that the operator $\mathcal{B}_1$ is dissipative,   result in  
\begin{multline}\label{dis-10}
\beta\gamma\mu_1\|A^\frac{\tau}{2}\Phi\|^2+\kappa\gamma\mu_2\|A^\frac{\sigma}{2} \Psi\|^2+\dfrac{\beta\kappa K^2}{\rho_3}\|A^\frac{\xi}{2}\Theta\|^2
= \text{Re}\dual{(i\lambda -\mathcal{B}_1)U}{U}\\= \text{Re}\dual{F}{U}\leq \|F\|_\mathcal{H}\|U\|_\mathcal{H}.
\end{multline}
On the other hand,  performing the duality product of \eqref{esp-20} for $\beta\gamma\rho_1 \phi$,    and remembering that the operators $A^\nu$ for all $\nu\in\mathbb{R}$ are
self-adjoint, resulting in
\begin{multline*}
\beta\kappa\gamma\dual{(\phi_x+\psi)}{\phi_x}= \beta\gamma\rho_1\|\Phi\|^2+\beta\gamma\rho_1\dual{\Phi}{f^1}\\
 -i\lambda\mu_1\beta\gamma\|A^\frac{\tau}{2}\phi\|^2
+\mu_1\beta\gamma\dual{A^\frac{\tau}{2} f^1}{A^\frac{\tau}{2} \phi}
+\beta\gamma\rho_1\dual{f^2}{\phi},
\end{multline*}
now performing the duality product of \eqref{esp-40} for $\beta\gamma\rho_2 \psi$,   and remembering that the operators $A^\nu$ for all $\nu\in\mathbb{R}$ are
self-adjoint, leading to
\begin{multline*}
\beta\kappa\gamma\dual{(\phi_x+\psi)}{\psi}+b\beta\gamma\|A^\frac{1}{2}\psi\|^2=\beta\gamma\rho_2\|\Psi\|^2
+\beta\gamma\rho_2\dual{\Psi}{f^3}\\
+\beta^2\gamma\dual{\Theta}{\psi_x}-i\lambda\beta\gamma\mu_2\|A^\frac{\sigma}{2}\psi\|^2+\beta\gamma\mu_2\dual{A^\frac{\sigma}{2}f^3}{A^\frac{\sigma}{2}\psi}+\beta\gamma\rho_2\dual{f^4}{\psi}.
\end{multline*}
Adding the last 2 equations, result in 
\begin{multline}\label{Eq001Exponential}
\beta\kappa\gamma\|\phi_x+\psi\|^2+b\beta\gamma\|A^\frac{1}{2}\psi\|^2=
\beta\gamma\rho_1\|\Phi\|^2+\beta\gamma\rho_2\|\Psi\|^2
-i\lambda\beta\gamma\{ \mu_1\|A^\frac{\tau}{2}u\|^2\\
+\mu_2\|A^\frac{\sigma}{2}\psi\|^2\}
 +\beta\gamma\rho_1\dual{\Phi}{f^1}+\beta\gamma\rho_2\dual{\Psi}{f^3}\\
+\beta\gamma\mu_1\dual{A^\frac{\tau}{2}f^1}{A^\frac{\tau}{2}\phi} +\beta\gamma\mu_2\dual{A^\frac{\sigma}{2}f^3}{A^\frac{\sigma}{2}\phi}\\
+\beta\gamma\rho_1\dual{f^2}{\phi}+\beta\gamma\rho_2\dual{f^4}{\psi}+\beta^2\gamma\dual{\Theta}{\psi_x}.
\end{multline}
Taking real part,  using norm $\|F\|_\mathcal{H}$ and $\|U\|_\mathcal{H}$ and     applying Cauchy-Schwarz and Young,  for $\varepsilon>0$ exists  $C_\varepsilon>0$ which does not depend on $\lambda$,  such that
\begin{multline*}
\beta\kappa\gamma\|\phi_x+\psi\|^2+b\beta\gamma\|A^\frac{1}{2}\psi\|^2 
\leq \beta\gamma\rho_1\|\Phi\|^2+\beta\gamma\rho_2\|\Psi\|^2+\varepsilon\|A^\frac{1}{2}\psi\|^2+C_\varepsilon\|\Theta\|^2
+C_\delta\|F\|_\mathcal{H}\|U\|_\mathcal{H}.
\end{multline*}
From  estimative \eqref{dis-10}  and the fact $0\leq\frac{\tau}{2}, \; 0\leq\frac{\sigma}{2}$  and $0\leq\frac{\xi}{2}$,  the continuous embedding $D(A^{\theta_2}) \hookrightarrow D(A^{\theta_1}),\;\theta_2>\theta_1$,  lead to
\begin{eqnarray}\label{Exponential002}
\beta\kappa\gamma\|\phi_x+\psi\|^2+b\beta\gamma\|A^\frac{1}{2}\psi\|^2&\leq& C_\delta\|F\|_\mathcal{H}\|U\|_\mathcal{H}.
\end{eqnarray}
On the other hand,  performing the duality product of \eqref{esp-60} for $\beta\kappa\rho_3 \theta$,    and remembering that the operators $A^\nu$ for all $\nu\in\mathbb{R}$ are
self-adjoint and from \eqref{esp-50},  result in 
\begin{multline*}
\beta\kappa\delta\|A^\frac{1}{2}\theta\|^2=\delta\beta\kappa\|\Theta\|^2+\beta\kappa\rho_3\dual{\Theta}{f^5}+\beta\kappa\gamma\dual{\Psi}{\theta_x}\\
-i\kappa\beta K\lambda\|A^\frac{\xi}{2}\theta\|^2+\kappa\beta K\dual{A^\frac{\xi}{2}f^5}{A^\frac{\xi}{2}\theta}+\beta\kappa\rho_3\dual{f^6}{\theta}
\end{multline*}
Taking real part and by using the inequalities Cauchy-Schwarz and Young and norms $\|F\|_\mathcal{H}$ and $\|U\|_\mathcal{H}$,  from $|\lambda|>1$,   for $\varepsilon>0$ exists $C_\varepsilon>0$ which does not depend on $\lambda$ such that
\begin{multline*}
\beta\kappa\delta\|A^\frac{1}{2}\theta\|^2
\leq C\{\|\Theta\|^2+\|\Theta\|\|f^5\|+\|A^\frac{\xi}{2}f^5\|\|A^\frac{\xi}{2}\theta\|+\|f^6\|\|\theta\|\}+C_\varepsilon\|\Psi\|^2+\varepsilon\|A^\frac{1}{2}\theta\|^2.
\end{multline*}
Then, from estimative \eqref{dis-10}  and the fact $0\leq\frac{\sigma}{2}$ and $0\leq\frac{\xi}{2}\leq\frac{1}{2}$  the continuous embedding $D(A^{\theta_2}) \hookrightarrow D(A^{\theta_1}),\;\theta_2>\theta_1$,  lead to
\begin{equation}\label{Exponential003}
\beta\kappa\delta\|A^\frac{1}{2}\theta\|^2\leq C_\delta\|F\|_\mathcal{H}\|U\|_\mathcal{H}.
\end{equation}
Finally,  from estimates \eqref{dis-10}, \eqref{Exponential002} and \eqref{Exponential003},   the proof of this lemma is completed.
\end{proof}
\begin{lemma}\label{Regularidad}
Let $\delta> 0$. There exists a constant $C_\delta > 0$ such that the solutions of the system \eqref{Eq05A}--\eqref{Eq07A} and \eqref{Eq08A}--\eqref{Eq09A}  for $|\lambda|\geq \delta$ satisfy the inequalities
\begin{eqnarray}\label{EqPrincipalLema}
\hspace*{-0.5cm}(i)\; |\lambda|\beta\gamma[\kappa\|\phi_x+\psi\|^2+b\|A^\frac{1}{2}\psi\|^2]\hspace*{-0.25cm} &\leq
&\hspace*{-0.25cm} |\lambda|\beta\gamma[\rho_1\|\Phi\|^2+\rho_2\|\Psi\|^2] +C_\delta\|F\|_\mathcal{H}\|U\|_\mathcal{H}.\\
\nonumber
\\
\label{EqPrincipalLemaA}
(ii)\; \beta\kappa\delta|\lambda|\|A^\frac{1}{2}\theta\|^2 &\leq & \beta\kappa\rho_3|\lambda|\|\Theta\|^2+ C_\delta\|F\|_\mathcal{H}\|U\|_\mathcal{H}.
\end{eqnarray}
\end{lemma}
\begin{proof}
{\bf Item $(i)$:} Performing the duality product of \eqref{esp-20} for $\beta\gamma\rho_1\lambda \phi$,   and remembering that the operators $A^\nu$ for all $\nu\in \mathbb{R} $ are self-adjunct,  result in 
\begin{eqnarray*}
\beta\kappa\gamma\lambda\dual{(\phi_x+\psi)}{\phi_x}& =&\beta\gamma\rho_1\lambda\|\Phi\|^2+\beta \gamma\rho_1\dual{\lambda \Phi}{f^1}-i\mu_1\beta\gamma\|A^\frac{\tau}{2}\Phi\|^2\\
& & -i\mu_1\beta\gamma\dual{A^\tau \Phi}{f^1}
+i\beta\gamma\rho_1\dual{f^2}{\Phi}+i\beta\gamma\rho_1\dual{f^2}{f^1}.
\end{eqnarray*}
Now performing the duality product of \eqref{esp-40} for $\beta\gamma\rho_2\lambda \psi$,  and remembering that the operators $A^\nu$ for all $\nu\in \mathbb{R} $ are
self-adjunct leads to 
\begin{eqnarray*}
\lambda\beta\gamma[\kappa\dual{(\phi_x+\psi)}{\psi}+
b\|A^\frac{1}{2}\psi\|^2]
\hspace*{-0.3cm}& = &\hspace*{-0.3cm} \beta\gamma\rho_2\lambda\|\Psi\|^2+\beta\gamma\rho_2\dual{\lambda \Psi}{f^3}+i\beta^2\gamma\dual{\Theta}{\psi_x} \\
& &\hspace*{-0.3cm} +i\beta^2\gamma\dual{\Theta}{f_x^3}  - i\mu_2\beta\gamma\|A^\frac{\sigma}{2}\Psi\|^2-i\mu_2\beta\gamma\dual{A^\sigma \Psi}{f^3}\\
& & \hspace*{-0.3cm} +i\beta\gamma\rho_2\dual{f^4}{\Psi} +i\beta\gamma\rho_2\dual{f^4}{f^3}.
\end{eqnarray*}
Adding the last 2 equations,  result in 
\begin{multline}\label{Eq001Gevrey}
\lambda\beta\gamma[\kappa\|\phi_x+\psi\|^2+b\|A^\frac{1}{2}\psi\|^2] =
\lambda\beta\gamma[\rho_1\|\Phi\|^2+\rho_2\|\Psi\|^2]
-i\beta\gamma\{\mu_1\|A^\frac{\tau}{2}\Phi\|^2\\+\mu_2\|A^\frac{\sigma}{2}\Psi\|^2\}+\beta\gamma\rho_1\dual{\lambda \Phi}{f^1}+\beta\gamma\rho_2\dual{\lambda \Psi}{f^3}-i\beta\gamma\{\mu_1 \dual{A^\tau \Phi}{f^1}\\
+\mu_2\dual{A^\sigma \Psi}{f^3} \}
+i\beta\gamma\rho_1\dual{f^2}{\Phi} +i\beta\gamma\rho_2\dual{f^4}{\Psi}
+i\beta\gamma\rho_1\dual{f^2}{f^1}\\
+i\beta\gamma\rho_2\dual{f^4}{f^3}+i\beta^2\gamma\dual{\Theta}{\psi_x}+i\beta^2\gamma\dual{\Theta}{f_x^3}.   
\end{multline}
On the other hand, from \eqref{esp-20}  and \eqref{esp-40},  result in 
\begin{multline}\label{Eq002Gevrey}
\beta\gamma\{\rho_1\dual{\lambda \Phi}{f^1}+\rho_2\dual{\lambda \Psi}{f^3}\}
 =i\beta\gamma\{\kappa\dual{(\phi_x+\psi)}{f^1_x}\\
 +\mu_1\dual{A^\tau \Phi}{f^1}-\rho_1\dual{f^2}{f^1}
 +b\dual{A^\frac{1}{2}\psi}{A^\frac{1}{2}f^3}+\kappa\dual{(\phi_x+\psi)}{f^3}\\
 -\beta\dual{\Theta}{f_x^3}+\mu_2\dual{A^\sigma \Psi}{f^3}-\rho_2\dual{f^4}{f^3}\}.
\end{multline}
Using the identity \eqref{Eq002Gevrey} in the \eqref{Eq001Gevrey} equation and simplifying,  lead to 
\begin{multline}\label{Eq003Gevrey}
\lambda\beta\gamma[\kappa\|\phi_x+\psi\|^2+b\|A^\frac{1}{2}\psi\|^2] =
\lambda\beta\gamma[\rho_1\|\Phi\|^2+\rho_2\|\Psi\|^2]
\\-i\{\mu_1\|A^\frac{\tau}{2}\Phi\|^2
+\mu_2\|A^\frac{\sigma}{2}\Psi\|^2\}
+i\beta\gamma\kappa\dual{\phi_x}{f_x^1}+i\beta\gamma\kappa\dual{\psi}{f_x^1}\\+ib\beta\gamma\dual{A^\frac{1}{2}\psi}{A^\frac{1}{2}f^3}
i\beta\gamma\kappa\dual{\phi_x}{f^3}+i\beta\gamma\kappa\dual{\psi}{f^3}\\
+i\beta\gamma\rho_1\dual{f^2}{\Phi}
+i\beta\gamma\rho_2\dual{f^4}{\Psi}+i\beta^2\gamma\dual{\Theta}{\psi_x}.
\end{multline}
As for $\varepsilon>0$  exists $C_\varepsilon>$ such that  $|i\beta^2\gamma\dual{\Theta}{\psi_x}|\leq \varepsilon\|A^\frac{1}{2}\psi\|^2+C_\varepsilon\|\Theta\|^2$, taking the real part,  and applying the Cauchy-Schwarz and Young inequalities,   estimative \eqref{EstimaEquivExp1}  of Lemma\eqref{LemaExponencial} and using the definitions of the $F$ and $U$ norm in \eqref{Eq003Gevrey}  the proof of item $(i)$  this lemma is completed.

On the other hand,  performing the duality product of \eqref{esp-60} for $\beta\kappa\rho_3\lambda \theta$,   and remembering that the operators $A^\theta$ for all $\theta\in \mathbb{R} $ are self-adjunct,  result in
\begin{multline}
\beta\kappa\delta\lambda\|A^\frac{1}{2}\theta\|^2=\beta\kappa\rho_3\dual{\lambda \Theta}{i\lambda\theta}-\beta\kappa\gamma\dual{\psi_x}{\lambda\theta}
+\beta\kappa K\dual{A^\xi\Theta}{\lambda\theta}+\beta\kappa\rho_3\dual{f^6}{\lambda\theta}.
\label{Eq000Ltheta}
\end{multline}
As
\begin{eqnarray}
\label{Eq001Ltheta}
\rho_3\dual{\lambda \Theta}{i\lambda\theta}&=&
\rho_3\lambda\|\Theta\|^2+i\delta\dual{A^\frac{1}{2}\theta}{A^\frac{1}{2}f^5}-i\gamma\dual{\Psi}{f^5_x} +i K\dual{A^\xi\Theta}{f^5} -i\rho_3\dual{f^6}{f^5}.\\
\label{Eq002Ltheta}
-\gamma\dual{\psi_x}{\lambda\theta} &=&-i\gamma\dual{\psi_x}{\Theta}-i\gamma\dual{\psi_x}{f^5}.\\
\label{Eq003Ltheta}
 K\dual{A^\xi\Theta}{\lambda\theta} &=& -i K\|A^\frac{\xi}{2}\Theta\|^2-i K\dual{A^\xi\Theta}{f^5}.
\\
\label{Eq004Ltheta}
\rho_3\dual{f^6}{\lambda\theta} &= & i\rho_3\dual{f^6}{\Theta}+i\rho_3\dual{f^6}{f^5}.
\end{eqnarray}
Using \eqref{Eq001Ltheta}--\eqref{Eq004Ltheta} in \eqref{Eq000Ltheta},  lead to
\begin{multline}\label{Eq005Ltheta}
\beta\kappa\delta\lambda\|A^\frac{1}{2}\theta\|^2 = \beta\kappa\rho_3\lambda\|\Theta\|^2+i\beta\kappa\delta\dual{A^\frac{1}{2}\theta}{A^\frac{1}{2}f^5}-i\beta\kappa\gamma\dual{\Psi}{f^5_x}-i\beta\kappa\gamma\dual{\psi_x}{\Theta}\\
 -i\beta\kappa\gamma\dual{\psi_x}{f^5}-i\beta\kappa K\|A^\frac{\xi}{2}\Theta\|^2
 +i\beta\kappa\rho_3\dual{f^6}{\Theta}.
\end{multline}
Taking real part of the equation \eqref{Eq005Ltheta} and using the inequalities Cauchy-Schwarz and Young and norms $\|F\|_\mathcal{H}$ and $\|U\|_\mathcal{H}$,  from $|\lambda|>1$,   for $\varepsilon>0$ exists $C_\varepsilon>0$ such that 
\begin{multline*}
\beta\kappa\delta|\lambda|\|A^\frac{1}{2}\theta\|^2 \leq\beta\kappa\rho_3|\lambda|\|\Theta\|^2+ C\{\|A^\frac{1}{2}\theta\|\|A^\frac{1}{2}f^5\|+\|\Psi\|\|f^5_x\|\\
+\|\psi_x\|^2+\|\Theta\|^2+\|\psi_x\|\|f^5\|+\|f^6\|\|\theta\| \}.
\end{multline*}
Then, from Lemma\eqref{LemaExponencial} and  as $\|\psi_x\|^2=\|A^\frac{1}{2}\psi\|^2$,   the proof of item $(ii)$ of this lemma is finished.
\end{proof}
\subsubsection{ Gevrey class of the first system}
\begin{definition}\label{Def1.1Tebou} Let $t_0\geq 0$ be a real number.  A strongly continuous semigroup $S(t)$, defined on a Banach space $ \mathcal{H}$, is of Gevrey class $s > 1$ for $t > t_0$, if $S(t)$ is infinitely differentiable for $t > t_0$, and for every compact set $K \subset (t_0,\infty)$ and each $\mu > 0$, there exists a constant $ C = C(\mu, K) > 0$ such that
	\begin{equation}\label{DesigDef1.1}
	||S^{(n)}(t)||_{\mathcal{L}( \mathcal{H})} \leq  C\mu ^n(n!)^s,  \text{ for all } \quad t \in K, n = 0,1,2...
	\end{equation}
\end{definition}
\begin{theorem}[\cite{TaylorM}]\label{Theorem1.2Tebon}
	Let $S(t)$  be a strongly continuous and bounded semigroup on a Hilbert space $ \mathcal{H}$. Suppose that the infinitesimal generator $\mathcal{B}$ of the semigroup $S(t)$ satisfies the following estimate, for some $0 < \eta < 1$:
	\begin{equation}\label{Eq1.5Tebon2020}
	\lim\limits_{|\lambda|\to\infty} \sup |\lambda |^\eta ||(i\lambda I-\mathcal{B})^{-1}||_{\mathcal{L}( \mathcal{H})} < \infty. 
	\end{equation}
	Then $S(t)$  is of Gevrey  class  $s$   for $t>0$, for every   $s>\dfrac{1}{\eta}$.
\end{theorem}
The main result of this subsection is as follows:
\begin{theorem} \label{GevreyLaminado} The  semigroup  $S(t)=e^{\mathcal{B}_1t}$   associated to the system \eqref{Eq05A}--\eqref{Eq07A} and \eqref{Eq08A}--\eqref{Eq09A}   is of Gevrey class $s>\frac{r+1}{2r}$  for  $r=\min \{\tau,\sigma,\xi\}$ for all
 $(\tau,\sigma,\xi )\in R_{CG}:= (0, 1)^3$.
\end{theorem}
\begin{proof}
From the resolvent  equation $F=(i\lambda I-\mathcal{B}_1)U$  for $\lambda\in\mathbb{R}$,  
$U=(i\lambda I-\mathcal{B}_1)^{-1}F$.  Furthermore,  show  \eqref{Eq1.5Tebon2020} of  theorem\eqref{Theorem1.2Tebon} it is enough to show,  that for $\varepsilon>0$ exists $C_\delta>0$,  such that:
$$\dfrac{|\lambda|^\frac{2r}{r+1} \|(i\lambda-\mathcal{B}_1)^{-1}F\|_\mathcal{H}}{\|F\|_\mathcal{H}}\leq C_\delta \Longleftrightarrow |\lambda |^\frac{2r}{r+1} \|U\|_\mathcal{H}\leq C_\delta \|F\|_\mathcal{H}, \qquad F\in\mathcal{H}\quad{\rm and}\quad  C_\delta>0.$$
Equivalent to
\begin{equation}\label{EstimaEquivalenteGevrey}
|\lambda |^\frac{2r}{r+1}\|U\|_\mathcal{H}^2\leq C_\delta\{\|F\|_\mathcal{H}\|U\|_\mathcal{H}\}\quad\rm{for}\quad \frac{2r}{r+1} \in(0,1).
\end{equation}
where  $r=\min\{\tau,\sigma,\xi\}$, \; for all $(\tau,\sigma,\xi)\in (0,1)^3$.

Next,  $[|\lambda|^\frac{2\tau}{1+\tau}\|\Phi\|+|\lambda|^\frac{2\sigma}{1+\sigma}\|\Psi\|+|\lambda|^\frac{2\xi}{1+\xi}\|\Theta\|]$  will be estimated.  \\
{\bf  Let's start by estimating the term $|\lambda|^\frac{2\tau}{1+\tau}\|\Phi\|$:}  It is  assume that   $|\lambda|>1$,  some ideas could be borrowed  from \cite{LiuR95}.  Set $\Phi=\Phi_1+\Phi_2$, where $\Phi_1\in D(A)$ and $\Phi_2\in D(A^0)$, with 
\begin{equation}\label{Eq110AnalyRR}
i\lambda \Phi_1+A \Phi_1=f^2, 
\hspace{3cm} i\lambda \Phi_2=\dfrac{\kappa}{\rho_1}A\phi+\dfrac{\kappa}{\rho_1}\psi_x-\dfrac{\mu_1}{\rho_1}A^\tau \Phi+A\Phi_1.
\end{equation} 
Firstly,  applying in the product duality  the first equation in \eqref{Eq110AnalyRR} by $\Phi_1$, then  by $A\Phi_1$    and recalling that the operator $A$  is
self-adjoint, resulting in 
\begin{equation}\label{Eq112AnalyRR}
|\lambda|\|\Phi_1\|+|\lambda|^\frac{1}{2}\|A^\frac{1}{2}\Phi_1\|+\|A\Phi_1\|\leq C\|F\|_\mathcal{H}.
\end{equation}
Applying the $A^{-\frac{1}{2}}$ operator on the second equation of \eqref{Eq110AnalyRR},   result in
\begin{equation*}
i\lambda A^{-\frac{1}{2}}\Phi_2= \dfrac{\kappa}{\rho_1}A^\frac{1}{2}\phi+\dfrac{\kappa}{\rho_1}A^{-\frac{1}{2}}\psi_x-\dfrac{\mu_1}{\rho_1}A^{\tau-\frac{1}{2}} \Phi+A^\frac{1}{2}\Phi_1,
\end{equation*}
 then,  as $\|A^{-\frac{1}{2}}\psi_x\|^2=\dual{-A^{-\frac{1}{2}}\psi_{xx}}{A^{-\frac{1}{2}}\psi}=\dual{A^\frac{1}{2}\psi}{A^{-\frac{1}{2}}\psi}=\|\psi\|^2$  and  $\tau-\frac{1}{2}\leq \frac{\tau}{2}$ taking into account the continuous embedding $D(A^{\theta_2}) \hookrightarrow D(A^{\theta_1}),\;\theta_2>\theta_1$,  result in 
\begin{equation}\label{Eq113AAnaly}
|\lambda|\|A^{-\frac{1}{2}} \Phi_2\|\leq C\{\|A^\frac{1}{2}\phi\|+ \|\psi\|+\|A^\frac{\tau}{2}\Phi\|\}+\|A^\frac{1}{2}\Phi_1\|,
\end{equation}
using  \eqref{EstimaEquivExp1}, \eqref{dis-10} and  estimative  \eqref{Eq112AnalyRR} ($\|A^\frac{1}{2}\Phi\|^2\leq C|\lambda|^{-1}\|F\|^2_\mathcal{H}$),  lead  to
\begin{equation*}
|\lambda|^2\|A^{-\frac{1}{2}}\Phi_2\|^2\leq C\{\|F\|_\mathcal{H}\|U\|_\mathcal{H} +|\lambda|^{-1}\|F\|^2_\mathcal{H} \}.
\end{equation*}
Equivalently 
\begin{equation}
\label{Eq113AnalyRR}
\|A^{-\frac{1}{2}}\Phi_2\|^2\leq C|\lambda|^{-\frac{4\tau+2}{\tau+1}}\{ |\lambda|^\frac{2\tau}{\tau+1}\|F\|_\mathcal{H}\|U\|_\mathcal{H}+\|F\|^2_\mathcal{H}\}\quad {\rm for}\quad 0\leq\tau\leq 1.
\end{equation}
On the  other hand, from $\Phi_2=\Phi-\Phi_1$,  \eqref{dis-10} and  as $\frac{\tau}{2}\leq\frac{1}{2}$ the inequality of \eqref{Eq112AnalyRR}, result in
\begin{eqnarray}
\nonumber
\|A^\frac{\tau}{2} \Phi_2\|^2 & \leq & C\{ \|A^\frac{\tau}{2} \Phi\|^2+\|A^\frac{\tau}{2}\Phi_1\|^2
\}\\
\label{Eq114AnalyRR}
 &\leq &  C\{\|F\|_\mathcal{H}\|U\|_\mathcal{H}+|\lambda|^{-1}\|F\|^2_\mathcal{H} \}\leq C|\lambda|^{-\frac{2\tau}{\tau+1}}\{ |\lambda|^\frac{2\tau}{\tau+1}\|F\|_\mathcal{H}\|U\|_\mathcal{H}+\|F\|^2_\mathcal{H}\}.
\end{eqnarray}
By Lions' interpolations inequality $0\in [-\frac{1}{2},\frac{\tau}{2}]$,  result in
\begin{equation}\label{Eq115AnalyRR}
 \|\Phi_2\|^2\leq C(\|A^{-\frac{1}{2}}\Phi_2\|^2)^\frac{\tau}{1+\tau}(\|A^\frac{\tau}{2}\Phi_2\|^2)^\frac{1}{1+\tau}.
\end{equation}
Then, using \eqref{Eq113AnalyRR} and \eqref{Eq114AnalyRR} in \eqref{Eq115AnalyRR}, for $ |\lambda|>1$,  result in 
\begin{equation}\label{Eq118AnalyRR}
 \|\Phi_2\|^2\leq C|\lambda|^{-\frac{4\tau}{1+\tau}}\{|\lambda|^\frac{2\tau}{\tau+1} \|F\|_\mathcal{H}\|U\|_\mathcal{H}+\|F\|^2_\mathcal{H}\}.
\end{equation}
Therefore,   as $\|\Phi\|^2\leq  \|\Phi_1\|^2+ \|\Phi_2\|^2$ from first inequality of  \eqref{Eq112AnalyRR},  \eqref{Eq118AnalyRR}, we get
\begin{equation*}
\|\Phi\|^2\leq C|\lambda|^{-2}\|F\|^2_\mathcal{H}+ C|\lambda|^{-\frac{4\tau}{1+\tau}}\{|\lambda|^\frac{2\tau}{\tau+1} \|F\|_\mathcal{H}\|U\|_\mathcal{H}+\|F\|^2_\mathcal{H}\},
\end{equation*}
 and  as for  $0\leq\tau\leq 1$ we have $|\lambda|^{-2}\leq |\lambda|^\frac{-4\tau}{1+\tau}$, result in 
\begin{equation}\label{Eq119AnalyRR}
 |\lambda|\|\Phi\|^2\leq C_\delta|\lambda|^\frac{1-3\tau}{1+\tau}\{|\lambda|^\frac{2\tau}{1+\tau}\|F\|_\mathcal{H}\|U\|_\mathcal{H}+\|F\|^2_\mathcal{H}\}\qquad\rm{for} \qquad 0\leq\tau\leq 1.
\end{equation}
{\bf On the other hand,   let's now estimate the missing term  $|\lambda|^\frac{2\sigma}{1+\sigma}\|\Psi\|$: } It is assumed that $|\lambda|>1$.  Set $\Psi=\Psi_1+\Psi_2$, where $\Psi_1\in D(A)$ and $\Psi_2\in D(A^0)$, with 
\begin{multline}\label{Eq110AnalyRRW}
i\lambda \Psi_1+A \Psi_1=f^4 \qquad {\rm and}\qquad i\lambda \Psi_2=-\dfrac{b}{\rho_2}A\psi -\dfrac{\kappa}{\rho_2}\phi_x-\dfrac{\kappa}{\rho_2}\psi-\dfrac{\beta}{\rho_2}\Theta_x-\dfrac{\mu_2}{\rho_2}A^\sigma \Psi+A\Psi_1.
\end{multline} 
Firstly,  applying in the product duality  the first equation in \eqref{Eq110AnalyRRW} by $\Psi_1$, then  by $A\Psi_1$    and recalling that the operator $A$  is
self-adjoint, resulting in 
\begin{equation}\label{Eq112AnalyRRW}
|\lambda|\|\Psi_1\|+|\lambda|^\frac{1}{2}\|A^\frac{1}{2}\Psi_1\|+\|A\Psi_1\|\leq C\|F\|_\mathcal{H}. 
\end{equation}
As follows from the second equation in \eqref{Eq110AnalyRRW}  that
\begin{equation*}
i\lambda A^{-\frac{1}{2}}\Psi_2= -\dfrac{b}{\rho_2}A^\frac{1}{2}\psi-\dfrac{\kappa}{\rho_2}A^{-\frac{1}{2}}\phi_x-\dfrac{\kappa}{\rho_2}A^{-\frac{1}{2}}\psi-\dfrac{\beta}{\rho_2}A^{-\frac{1}{2}}\Theta_x-\dfrac{\mu_2}{\rho_2}A^{\sigma-\frac{1}{2}}\Psi+A^\frac{1}{2}\Psi_1,
\end{equation*}
 then,  as $\|A^{-\frac{1}{2}}\phi_x\|^2=\dual{-A^{-\frac{1}{2}}\phi_{xx}}{A^{-\frac{1}{2}}\phi}=\|\phi\|^2$,  $\|A^{-\frac{1}{2}}\Theta_x\|^2=\|\Theta\|^2 $ and  $\sigma-\frac{1}{2}\leq \frac{\sigma}{2}$ taking into account the continuous embedding $D(A^{\theta_2}) \hookrightarrow D(A^{\theta_1}),\;\theta_2>\theta_1$,  lead to
\begin{equation}\label{Eq113AAnalyRRW}
|\lambda|\|A^{-\frac{1}{2}} \Psi_2\|\leq C\{ \|A^\frac{1}{2}\phi\|+\|\psi\|+\|\Theta\|+\|A^\frac{\sigma}{2}\Psi\|\}+\|A^\frac{1}{2}\Psi_1\|
\end{equation}
Using estimative \eqref{EstimaEquivExp1} and   estimative  \eqref{Eq112AnalyRRW},  for $|\lambda|>1$,  result in
\begin{equation*}
\|A^{-\frac{1}{2}}\Psi_2\|^2 \leq  C|\lambda|^{-3}\{|\lambda|\|F\|_\mathcal{H}\|U\|_\mathcal{H}+\|F\|^2_\mathcal{H}\}.
\end{equation*}
Equivalently
\begin{equation}\label{Eq113AnalyRRW}
\|A^{-\frac{1}{2}}\Psi_2\|^2 \leq C|\lambda|^{-\frac{4\sigma+2}{\sigma+1}}\{ |\lambda|^\frac{2\sigma}{\sigma+1}\|F\|_\mathcal{H}\|U\|_\mathcal{H}+\|F\|^2_\mathcal{H}\}\quad {\rm for}\quad 0\leq \sigma\leq 1.
\end{equation}
On the  other hand, from $\Psi_2=\Psi-\Psi_1$,  \eqref{dis-10} and  as $\frac{\sigma}{2}\leq\frac{1}{2}$ the  inequality  \eqref{Eq112AnalyRRW},  result in  
\begin{eqnarray}
\nonumber
\|A^\frac{\sigma}{2} \Psi_2\|^ 2& \leq & C\{ \|A^\frac{\sigma}{2} \Psi\|^2+\|A^\frac{\sigma}{2}\Psi_1\|^2\}\\
\label{Eq114AnalyRRW}
& \leq &   C\{\|F\|_\mathcal{H}\|U\|_\mathcal{H}+|\lambda|^{-1}\|F\|^2_\mathcal{H}\}
\leq |\lambda|^{-\frac{2\sigma}{\sigma+1}}\{|\lambda|^\frac{2\sigma}{\sigma+1}\|F\|_\mathcal{H}\|U\|_\mathcal{H}+\|F\|^2_\mathcal{H}\}.
\end{eqnarray}
Now,  by Lions' interpolations inequality $0\in [-\frac{1}{2},\frac{\sigma}{2}]$,  lead to
\begin{equation}\label{Eq115AnalyRRW}
 \|\Psi_2\|^2 \leq C (\|A^{-\frac{1}{2}}\Psi_2\|^2)^\frac{\sigma}{1+\sigma}(\|A^\frac{\sigma}{2}\Psi_2\|^2)^\frac{1}{1+\sigma}.
\end{equation}
Then, using \eqref{Eq113AnalyRRW} and \eqref{Eq114AnalyRRW} in \eqref{Eq115AnalyRRW}, for $ |\lambda|>1$,  result in 
\begin{equation}\label{Eq118AnalyRRW}
 \|\Psi_2\|^2 \leq C|\lambda|^{\frac{-4\sigma}{(1+\sigma)}}\{|\lambda|^\frac{2\sigma}{\sigma+1} \|F\|_\mathcal{H}\|U\|_\mathcal{H}+\|F\|_\mathcal{H}\}.
\end{equation}
Therefore,   as $\|\Psi\|^2\leq  \|\Psi_1\|^2+ \|\Psi_2\|^2$ from first inequality of  \eqref{Eq112AnalyRRW}, \eqref{Eq118AnalyRRW} and $|\lambda|^{-2}\leq|\lambda|^\frac{-4\sigma}{1+\sigma}$,  result in 
\begin{equation}\label{Eq119AnalyRRW}
|\lambda|\|\Psi\|^2\leq |\lambda|^\frac{1-3\sigma}{1+\sigma}\{|\lambda|^\frac{2\sigma}{\sigma+1}\|F\|_\mathcal{H}\|U\|_\mathcal{H}+\|F\|^2_\mathcal{H}\}\qquad\rm{for}\qquad 0\leq\sigma\leq 1.
\end{equation}
{\bf  Finally,   let's now estimate the missing term  $|\lambda|^\frac{2\xi}{1+\xi}\|\Theta\|$:}   Now we assume  $|\lambda|>1$.  Set $\Theta=\Theta_1+\Theta_2$, where $\Theta_1\in D(A)$ and $\Theta_2\in D(A^0)$, with 
\begin{equation}\label{Eq110AnalyRRT}
i\lambda \Theta_1+A \Theta_1=f^6 
\qquad{\rm and}\qquad 
 i\lambda \Theta_2=-\dfrac{\delta}{\rho_3}A\theta -\dfrac{\gamma}{\rho_3}\Psi_x-\dfrac{K}{\rho_3}A^\xi \Theta+A\Theta_1.
\end{equation} 
Firstly,  applying in the product duality  the first equation in \eqref{Eq110AnalyRRT} by $\Theta_1$, then  by $A\Theta_1$    and recalling that the operator $A$  is
self-adjoint, resulting in 
\begin{equation}\label{Eq112AnalyRRT}
|\lambda|\|\Theta_1\|+|\lambda|^\frac{1}{2}\|A^\frac{1}{2}\Theta_1\|+\|A\Theta_1\|\leq C\|F\|_\mathcal{H}. 
\end{equation}
As follows from the second equation in \eqref{Eq110AnalyRRT}  that
\begin{equation*}
i\lambda A^{-\frac{1}{2}}\Theta_2= -\dfrac{\delta}{\rho_3}A^\frac{1}{2}\theta -\dfrac{\gamma}{\rho_3}A^{-\frac{1}{2}}\Psi_x-\dfrac{K}{\rho_3}A^{\xi-\frac{1}{2}} \Theta+A^\frac{1}{2}\Theta_1,
\end{equation*}
 then,  as  $\|A^{-\frac{1}{2}}\Psi_x\|^2=\|\Psi\|^2 $ and  $-\frac{1}{2}\leq 0$,  $\xi-\frac{1}{2}\leq \frac{\xi}{2}$ taking into account the continuous embedding $D(A^{\theta_2}) \hookrightarrow D(A^{\theta_1}),\;\theta_2>\theta_1$,  result in 
\begin{equation}\label{Eq113AAnalyRRT}
|\lambda|\|A^{-\frac{1}{2}} \Theta_2\|\leq C\{ \|A^\frac{1}{2}\theta\|+\|A^\frac{\xi}{2}\Theta\|+\|\Psi\|\}+\|A^\frac{1}{2}\Theta_1\|
\end{equation}

Using estimates \eqref{EstimaEquivExp1}, \eqref{dis-10} and   estimative  \eqref{Eq112AnalyRRT},  for $|\lambda|>1$,  lead to
\begin{equation}\label{Eq113AnalyRRT}
\|A^{-\frac{1}{2}}\Theta_2\|^2\leq C|\lambda|^{-\frac{4\xi+2}{\xi+1}}\{|\lambda|^\frac{2\xi}{\xi+1}\|F\|_\mathcal{H}\|U\|_\mathcal{H}+\|F\|^2_\mathcal{H}\}.
\end{equation}
On the  other hand, from $\Theta_2=\Theta-\Theta_1$,  \eqref{dis-10} and  as $\frac{\xi}{2}\leq\frac{1}{2}$ the inequality  \eqref{Eq112AnalyRRT},   result in 
\begin{eqnarray}
\nonumber
\|A^\frac{\xi}{2} \Theta_2\|^2 & \leq & C\{ \|A^\frac{\xi}{2} \Theta\|^2+\|A^\frac{\xi}{2}\Theta_1\|^2\}\\
\label{Eq114AnalyRRT}
& \leq &   C|\lambda|^{-\frac{2\xi}{\xi+1}}\{|\lambda|^\frac{2\xi}{\xi+1}\|F\|_\mathcal{H}\|U\|_\mathcal{H}+\|F\|^2_\mathcal{H}\}.
\end{eqnarray}
Now,  by Lions' interpolations inequality $0\in [-\frac{1}{2},\frac{\xi}{2}]$,   result in 
\begin{equation}\label{Eq115AnalyRRT}
\|\Theta_2\|^2\leq C (\|A^{-\frac{1}{2}}\Theta_2\|^2)^\frac{\xi}{1+\xi}(\|A^\frac{\xi}{2}\Theta_2\|^2)^\frac{1}{1+\xi}.
\end{equation}
Then, using \eqref{Eq113AnalyRRT} and \eqref{Eq114AnalyRRT} in \eqref{Eq115AnalyRRT}, for $ |\lambda|>1$,  result in 
\begin{equation}\label{Eq118AnalyRRT}
 \|\Theta_2\|^2\leq C|\lambda|^{\frac{-4\xi}{(1+\xi)}}\{|\lambda|^\frac{2\xi}{\xi+1} \|F\|_\mathcal{H}\|U\|_\mathcal{H}+\|F\|_\mathcal{H}\}.
\end{equation}
Therefore,   as $\|\Theta\|^2\leq \|\Theta_1\|^2+ \|\Theta_2\|^2$ from first inequality of  \eqref{Eq112AnalyRRT}, estimative \eqref{Eq118AnalyRRT} and $|\lambda|^{-2}\leq |\lambda|^\frac{-4\xi}{1+\xi}$,   result in 
\begin{equation}\label{Eq119AnalyRRT} 
|\lambda|\|\Theta\|^2\leq C_\delta |\lambda|^\frac{1-3\xi}{1+\xi}\{|\lambda|^\frac{2\xi}{\xi+1}\|F\|_\mathcal{H}\|U\|_\mathcal{H}+\|F\|^2_\mathcal{H}\}\qquad \rm{for}\qquad 0\leq\xi\leq 1.
\end{equation}
Using the estimates  \eqref{Eq119AnalyRR}   and \eqref{Eq119AnalyRRW}  in the inequality item $(i)$ of Lemma \ref{Regularidad},   result in 
\begin{multline}\label{Eq020GevreyA}
\beta\gamma\kappa\|\phi_x+\psi\|^2+\beta\gamma b\|A^\frac{1}{2}\psi\|^2
 \leq
C_\delta\big\{ |\lambda|^\frac{-2\tau}{1+\tau}+|\lambda|^\frac{-2\sigma}{1+\sigma}+|\lambda|^{-1}\big\}\|F\|_\mathcal{H}\|U\|_\mathcal{H}\\
+\big\{|\lambda|^{-\frac{4\tau}{\tau+1}}+|\lambda|^{-\frac{4\sigma}{\sigma+1}}\big\}\|F\|^2_\mathcal{H}. 
\end{multline}
As $ \tau\leq \sigma \Longrightarrow -\frac{\tau}{\tau+1}\geq -\frac{\sigma}{\sigma+1}$\quad  and \quad  $ \tau\geq \sigma \Longrightarrow -\frac{\tau}{\tau+1}\leq -\frac{\sigma}{\sigma+1}$, we obtain
\begin{multline}\label{Eq020Gevrey}
|\lambda|\{ \beta\gamma\kappa\|\phi_x+\psi\|^2+\beta\gamma b\|A^\frac{1}{2}\psi\|^2
  \}\\
  \leq C_\delta\left\{ \begin{array}{ccc}
  |\lambda|^\frac{1-3\tau}{\tau+1}\{|\lambda|^\frac{2\tau}{\tau+1} \|F\|_\mathcal{H}\|U\|_\mathcal{H}+\|F\|^2_\mathcal{H}\} & \text{for} & \tau\leq \sigma\\\\
 |\lambda|^\frac{1-3\sigma}{\sigma+1}\{ |\lambda|^\frac{2\sigma}{\sigma+1}\|F\|_\mathcal{H}\|U\|_\mathcal{H}+\|F\|^2_\mathcal{H} \}   &\text{for}  & \sigma\leq \tau
\end{array}   \right.
\end{multline}
Furthermore, using estimative \eqref{Eq119AnalyRRT} in  estimate \eqref{EqPrincipalLemaA} (item $(ii)$ of the Lemma \ref{Regularidad}) and as $\frac{1-3\xi}{\xi+1}+\frac{2\xi}{\xi+1}=\frac{1-\xi}{\xi+1}>0$, we have
\begin{multline}
\label{Eq021Gevrey}
\beta\kappa\delta |\lambda| \|A^\frac{1}{2}\theta\|^2 \leq \beta\kappa\rho_3|\lambda| \|\Theta\|^2+C_\delta|\|F\|_\mathcal{H}\|U\|_\mathcal{H}
\leq C_\delta|\lambda|^\frac{1-3\xi}{1+\xi}\{|\lambda|^\frac{2\xi}{1+\xi}\|F\|_\mathcal{H}\|U\|_\mathcal{H}+\|F\|^2_\mathcal{H}\}. 
\end{multline}

Finally summing the estimates \eqref{Eq119AnalyRR},\eqref{Eq119AnalyRRW}, \eqref{Eq119AnalyRRT},\eqref{Eq020Gevrey} and \eqref{Eq021Gevrey},  we have
\begin{equation*}
|\lambda|\|U\|^2_\mathcal{H}\leq C_\delta\left\{  \begin{array}{ccc}
 |\lambda|^\frac{1-3\tau}{1+\tau}\{ |\lambda|^\frac{2\tau}{\tau+1}\|F\|_\mathcal{H}\|U\|_\mathcal{H}+\|F\|^2_\mathcal{H}\} &\text{for} & \tau\leq \sigma\quad{\rm and}\quad \tau\leq \xi, 
 \\\\
  |\lambda|^\frac{1-3\sigma}{1+\sigma}\{ |\lambda|^\frac{2\sigma}{\sigma+1}\|F\|_\mathcal{H}\|U\|_\mathcal{H}+\|F\|^2_\mathcal{H}\} &\text{for} & \sigma \leq \tau\quad{\rm and}\quad \sigma\leq \xi,\\\\
   |\lambda|^\frac{1-3\xi}{1+\xi}\{ |\lambda|^\frac{2\xi}{\xi+1}\|F\|_\mathcal{H}\|U\|_\mathcal{H}+\|F\|^2_\mathcal{H}\} &\text{for} & \xi\leq \sigma\quad{\rm and}\quad \xi\leq \tau.
\end{array}\right.
\end{equation*} for $(\tau,\sigma,\xi)\in [0,1]^3$, equivalently
\begin{equation*}
\|U\|^2_\mathcal{H}\leq C_\delta\left\{  \begin{array}{ccc}
 |\lambda|^\frac{-4\tau}{1+\tau}\{ |\lambda|^\frac{2\tau}{\tau+1}\|F\|_\mathcal{H}\|U\|_\mathcal{H}+\|F\|^2_\mathcal{H}\} &\text{for} & \tau\leq \sigma\quad{\rm and}\quad \tau\leq \xi, 
 \\\\
  |\lambda|^\frac{-4\sigma}{1+\sigma}\{ |\lambda|^\frac{2\sigma}{\sigma+1}\|F\|_\mathcal{H}\|U\|_\mathcal{H}+\|F\|^2_\mathcal{H}\} &\text{for} & \sigma \leq \tau\quad{\rm and}\quad \sigma\leq \xi,\\\\
   |\lambda|^\frac{-4\xi}{1+\xi}\{ |\lambda|^\frac{2\xi}{\xi+1}\|F\|_\mathcal{H}\|U\|_\mathcal{H}+\|F\|^2_\mathcal{H}\} &\text{for} & \xi\leq \sigma\quad{\rm and}\quad \xi\leq \tau.
\end{array}\right.
\end{equation*}
Finally, applying Young inequality, we have 
\begin{equation*}
\left\{  \begin{array}{ccc}
 |\lambda|^\frac{2\tau}{\tau+1}\|U\|_\mathcal{H}\leq C_\delta|\|F\|_\mathcal{H} &\text{for} & \tau\leq \sigma\quad{\rm and}\quad \tau\leq \xi, 
 \\\\
  |\lambda|^\frac{2\sigma}{1+\sigma}\|U\|_\mathcal{H}\leq C_\delta\|F\|_\mathcal{H} &\text{for} & \sigma \leq \tau\quad{\rm and}\quad \sigma\leq \xi,\\\\
   |\lambda|^\frac{2\xi}{1+\xi}\|U\|_\mathcal{H}\leq C_\delta \|F\|_\mathcal{H} & \text{for} & \xi\leq \sigma\quad{\rm and}\quad \xi\leq \tau.
\end{array}\right.
\end{equation*}
Therefore,  proof of this theorem is finished.
\end{proof}

\subsubsection{Analyticity of $S(t)=e^{\mathcal{B}_1t}$  for $(\tau,\sigma,\xi)\in \big[\frac{1}{2},  1\big]^3$} \label{3.1.2}
Next, we announce the Liu-Zheng Book Theorem that will be applied in the proof of the analyticity of $S(t)=e^{\mathcal{B}_it}$ for $i=1, 2$ in their respective regions.
\begin{theorem}[see \cite{LiuZ}]\label{LiuZAnalyticity}
	Let $S(t)=e^{\mathcal{B}t}$ be $C_0$-semigroup of contractions  on a Hilbert space $ \mathcal{H}$.  Suppose that
	\begin{equation}\label{EixoImaginary}
	\rho(\mathcal{B})\supseteq\{ i\lambda/ \lambda\in \R \} 	\equiv i\R
	\end{equation}
	 Then $S(t)$ is analytic if and only if
	\begin{equation}\label{Analiticity}
	 \limsup\limits_{|\lambda|\to
		\infty}
	\|\lambda(i\lambda I-\mathcal{B})^{-1}\|_{\mathcal{L}( \mathcal{H})}<\infty
	\end{equation}
	holds.
\end{theorem}
Before proving the main result of this section,  the following two lemmas will be proved.
\begin{lemma}\label{EixoImaginary01}
The $C_0-$semigroup  of contractions $S(t)=e^{\mathcal{B}_1t}$ satisfies the condition 
 \begin{equation}\label{EixoImaginary01A}
	\rho(\mathcal{B}_1)\supseteq\{ i\lambda/ \lambda\in \R \} 	\equiv i\R.
	\end{equation}
\end{lemma}
\begin{proof}
Let's prove $i\R\subset\rho( \mathcal{B}_1)$  by contradiction,  it is  supposed  that $i\R\not\subset \rho( \mathcal{B}_1)$. As $0\in\rho( \mathcal{B}_1)$ and  $\rho( \mathcal{B}_1)$ is open,  the highest positive number $\lambda_0$  is considered  such that the interval  $(-i\lambda_0,i\lambda_0)\subset\rho( \mathcal{B}_1)$,  then $i\lambda_0$ or $-i\lambda_0$ is an element of the spectrum $\sigma( \mathcal{B}_1)$.  It is supposed that $i\lambda_0\in \sigma( \mathcal{B}_1)$ (if $-i\lambda_0\in \sigma( \mathcal{B}_1)$ the proceeding is similar). Then, for $0<\nu<\lambda_0$ there exist a sequence of real numbers $(\lambda_n)$, with $0<\nu\leq\lambda_n<\lambda_0$, $\lambda_n\to \lambda_0$, and a vector sequence  $U_n=(\phi_n,\Phi_n,\psi_n,\Psi_n, \theta_n,\Theta_n)\in \mathcal{D}( \mathcal{B}_1)$ with  unitary norms, such that
\begin{eqnarray*}
\|(i\lambda_n- \mathcal{B}_1) U_n\|_\mathcal{H}=\|F_n\|_\mathcal{H}\to 0,
\end{eqnarray*}
as $n\to \infty$. From \eqref{Exponential002} for $0\leq\tau\leq 1$ and $0\leq\sigma\leq 1$,  result in  
\begin{eqnarray*}
\beta\kappa\gamma\|\phi_{xn}+\psi_n\|^2+b\beta\gamma\|A^\frac{1}{2}\psi_n\|^2&\leq& C_\delta\|F_n\|_\mathcal{H}\|U_n\|_\mathcal{H}.
\end{eqnarray*}
From \eqref{Exponential003},   for $0\leq\xi\leq 1$,  result in 
\begin{equation}\label{Exponential003n}
\beta\kappa\delta\|A^\frac{1}{2}\theta_n\|^2\leq C_\delta\|F_n\|_\mathcal{H}\|U_n\|_\mathcal{H}.
\end{equation}
In addition to the estimative \eqref{dis-10},  for $0\leq\tau\leq 1$ and $0\leq\sigma\leq 1$,  result in  
\begin{eqnarray*}
\beta\gamma\mu_1\|A^\frac{\tau}{2}\Phi_n\|^2+\kappa\gamma\mu_2\|A^\frac{\sigma}{2}\Psi_n\|^2+\dfrac{\beta\kappa K}{\rho_3}\|A^\frac{\xi}{2}\Theta_n\|^2\leq C_\delta \|F_n\|_\mathcal{H}\|U_n\|_\mathcal{H}.
\end{eqnarray*}
Consequently,   $\|U_n\|_\mathcal{H}^2 \to 0.$
Therefore,  lead to  $\|U_n\|_\mathcal{H}\to  0$ but this is absurd,  since $\|U_n\|_\mathcal{H}=1$ for all $n\in\N$. Thus, $i\R\subset \rho(\mathcal{B}_1)$. This completes the proof.
\end{proof}
\begin{lemma}\label{Lema001Analiticity}
Let $\delta> 0$. There exists a constant $C_\delta > 0$ such that the solutions of \eqref{Eq01B2}--\eqref{Eq03B2} and  \eqref{Eq08A}--\eqref{Eq09A}
for $|\lambda|\geq  \delta$  satisfy the inequality
\begin{eqnarray}\label{Eq004Lema02A}
(i)\quad |\lambda|\|\Phi\|^2\leq  C_\delta\|F\|_\mathcal{H}\|U\|_\mathcal{H}\qquad{\rm for}\qquad \frac{1}{2}\leq\tau\leq 1.\\
\label{Eq005Lema02A}
(ii)\quad |\lambda|\|\Psi\|^2\leq  C_\delta\|F\|_\mathcal{H}\|U\|_\mathcal{H}\qquad{\rm for}\qquad \frac{1}{2}\leq\sigma\leq 1.\\
\label{Eq006Lema02A}
(iii)\quad |\lambda|\|\Theta\|^2\leq C_\delta\|F\|_\mathcal{H}\|U\|_\mathcal{H}\qquad{\rm for}\qquad  \frac{1}{2}\leq\xi\leq 1.
\end{eqnarray}
\end{lemma}
\begin{proof}
{\bf Item $(i)$:}  of this lemma will be initially shown,  performing the duality product of \eqref{esp-20} for $\dfrac{\rho_1}{\mu_1}A^{-\tau}\lambda \Phi$,  and recalling that the operator $A^\nu$ is self-adjoint for all $\nu\in\mathbb{R}$,  results  in 
\begin{eqnarray*}
\lambda\|\Phi\|^2\hspace*{-0.3cm} &=& \dfrac{\kappa}{\mu_1}\dual{\lambda(-A\phi+\psi_x)}{A^{-\tau} \Phi}+\dfrac{\rho_1}{\mu_1}\dual{f^2}{\lambda A^{-\tau}\Phi}-i\dfrac{\rho_1}{\mu_1}\lambda^2\|A^{-\frac{\tau}{2}}\Phi\|^2\\
&=&\hspace*{-0.3cm}   i\dfrac{\kappa}{\mu_1}\|A^\frac{1-\tau}{2} \Phi\|^2+i\dfrac{\kappa}{\mu_1}\dual{A^\frac{1}{2}f^1}{A^{\frac{1}{2}-\tau}\Phi}+i\dfrac{\kappa}{\mu_1}\dual{\Psi}{A^{-\tau}\Phi_x}-i\dfrac{\kappa}{\mu_1}\dual{f_x^3}{A^{-\tau}\Phi}\\
& &-i\dual{f^2}{\Phi} -i\dfrac{\kappa}{\mu_1}\dual{f^2}{A^{1-\tau}\phi} +i\dfrac{\kappa}{\mu_1}\dual{A^{-\tau}f^2}{\psi_x}  +i\dfrac{\rho_1}{\mu_1}\|A^{-\frac{\tau}{2}}f^2\|^2-i\dfrac{\rho_1}{\mu_1}\lambda^2\|A^{-\frac{\theta}{2}}\Phi\|^2.
\end{eqnarray*}
Noting that: $\|A^{-\tau} \Phi_x\|^2=\dual{A^{-\tau}\Phi_x}{A^{-\tau}\Phi_x}=\dual{-A^{-\tau}\Phi_{xx}}{A^{-\tau}\Phi}=\|A^\frac{1-2\tau}{2} \Phi\|^2$,  taking real part  and considering that  $\frac{1}{2}\leq\tau\leq 1$  using estimative \eqref{dis-10} and using Cauchy-Schwarz and Young  inequalities,  for $\varepsilon>0$ exists $C_\varepsilon$,   such that
\begin{eqnarray*}
|\lambda|\|\Phi\|^2 & \leq & C_\delta\|F\|_\mathcal{H}\|U\|_\mathcal{H}+\varepsilon\|A^{-\tau}\Phi_x\|^2+C_\varepsilon\|\Psi\|^2.
\end{eqnarray*}
As $0\leq\frac{\sigma}{2}$,  then from  estimative  \eqref{dis-10}  $\|\Psi\|^2\leq C_\delta\|F\|_\mathcal{H}\|U\|_\mathcal{H}$.  From the continuous embedding for  $|\lambda|\geq 1$,   the proof of item $(i)$ of this lemma is finished.

{\bf Item $(ii)$:}  Again similarly,  performing the duality product of \eqref{esp-40} for $\dfrac{\rho_2}{\mu_2}A^{-\sigma}\lambda \Psi$, using \eqref{esp-30}, and recalling the                                        self-adjointness of $A^\nu$,   $\nu \in\mathbb{R}$,  leads to 
\begin{eqnarray*}
\lambda\|\Psi\|^2 &=& i\dfrac{b}{\mu_2}\|A^\frac{1-\sigma}{2}\Psi\|^2+i\dfrac{b}{\mu_2}\dual{A^\frac{1}{2}f^3}{A^{\frac{1}{2}-\sigma}\Psi}+i\dfrac{\kappa}{\mu_2}\dual{A^{-\sigma}\Phi_x}{\Psi}+i\dfrac{\kappa}{\mu_2}\dual{f^1_x}{A^{-\sigma} \Psi}\\
& & +i\dfrac{\kappa}{\mu_2}\|A^{-\frac{\sigma}{2}}\Psi\|^2+i\dfrac{\kappa}{\mu_2}\dual{f^3}{A^{-\sigma}\Psi}-i\dfrac{b}{\rho_2}\dual{f^4}{A^{1-\sigma}\psi}-i\dfrac{\kappa}{\mu_2}\dual{f^4}{A^{-\sigma} \phi_x}\\
& & -i\dfrac{\kappa}{\mu_2}\dual{f^4}{A^{-\sigma}\psi}-i\dfrac{\beta}{\mu_2}\dual{f^4}{A^{-\sigma}\Theta}-i\dual{f^4}{\Psi} +   i\dfrac{\rho_2}{\mu_2}\|A^{-\frac{\sigma}{2}}f^4\|^2 -i\dfrac{\rho_2}{\mu_2}\lambda^2\|A^{-\frac{\sigma}{2}}\Psi\|^2.
\end{eqnarray*}
Noting that: $\|A^\frac{1-2\sigma}{2}\Phi\|^2=\|A^{-\sigma} \Phi_x\|^2$, taking real part,    considering $\frac{1}{2}\leq\sigma\leq 1$,    using Cauchy-Schwarz and Young  inequalities,  for $\varepsilon>0$ exists $C_\varepsilon$  such that
\begin{eqnarray*}
|\lambda|\|\Psi\|^2 & \leq & C_\delta\|F\|_\mathcal{H}\|U\|_\mathcal{H}+C_\varepsilon\|A^\frac{1-2\sigma}{2}\Phi\|^2+\varepsilon\|\Psi\|^2.
\end{eqnarray*}
As $-\frac{1-2\sigma}{2} \leq 0\Longleftrightarrow \frac{1}{2}\leq \sigma\leq 1$.   Considering  $|\lambda|\geq 1$,  using  estimative  \eqref{EstimaEquivExp1},    norms $\|F\|_\mathcal{H}$ and $\|U\|_\mathcal{H}$,    the proof of item $(ii)$ of this lemma is finished.

{\bf  Item $(iii)$:}  Finally,  let's prove the $(iii)$ item of this lemma, performing the duality product of \eqref{esp-60} for $\dfrac{\rho_3}{K}A^{-\xi}\lambda \Theta$, using \eqref{esp-30}, and recalling the   self-adjointness of $A^\nu$,   $\nu \in\mathbb{R}$,  result in 
\begin{eqnarray*}
\lambda\|\Theta\|^2 & = & -i\lambda^2\dfrac{\rho_3}{K}\|A^{-\frac{\xi}{2}}\Theta\|^2+i\dfrac{\delta}{K}\|A^\frac{1-\xi}{2}\Theta\|^2+i\dfrac{\delta}{K}\dual{A^\frac{1}{2}f^5}{A^{\frac{1}{2}-\xi}\Theta}-i\dfrac{\delta}{K}\dual{f^6}{A^{1-\xi}\theta}\\
& & +\dfrac{\gamma}{K}\dual{\sqrt{|\lambda|}\Psi}{\dfrac{\lambda}{\sqrt{|\lambda|}}A^{-\xi}\Theta_x} -i\dfrac{\gamma}{K}\dual{f^6}{A^{-\xi}\Psi_x}-i\dual{f^6}{\Theta}+i\dfrac{\rho_3}{K}\|A^{-\frac{\xi}{2}}f^6\|^2.
\end{eqnarray*}
Taking the real part,   applying Cauchy-Schwarz and Young inequalities,  for $\varepsilon>0$ exists $C_\varepsilon>0$ such that
\begin{multline}
|\lambda|\|\Theta\|^2\leq C\{ \|A^\frac{1-\xi}{2}\Theta\|^2+\|A^\frac{1}{2}f^5\|\|A^{\frac{1}{2}-\xi}\Theta\|\\
+\|f^6\|\|A^{1-\xi}\theta\|+\|f^6\|\|A^{-\xi}\Psi_x\|\}
+C_\varepsilon |\lambda|\|\Psi\|^2+\varepsilon |\lambda|\|A^{-\xi}\Theta_x\|^2.
\end{multline}
Considering $\frac{1}{2}\leq\xi\leq 1$,   $\frac{1-2\xi}{2}\leq 0$,   and remembering the identities $\|A^{-\xi}\Psi_x\|^2=\|A^\frac{1-2\xi}{2}\Psi\|^2$,  $\|A^{-\xi}\Theta_x\|^2=\|A^\frac{1-2\xi}{2}\Theta\|^2$  in addition to  using the item $(ii)$ of this lemma the proof of the item $(iii)$ of this lemma is finished.
\end{proof}
The main result of this subsection is the following theorem
\begin{theorem}
The semigroup $S(t)=e^{\mathcal{B}_1t}$ associated to the system \eqref{Eq05A}--\eqref{Eq07A} and \eqref{Eq08A}--\eqref{Eq09A} is analytic when the three parameters $\tau$, $\sigma$ and $\xi$ vary in the interval $[\frac{1}{2},1]$.
\end{theorem}
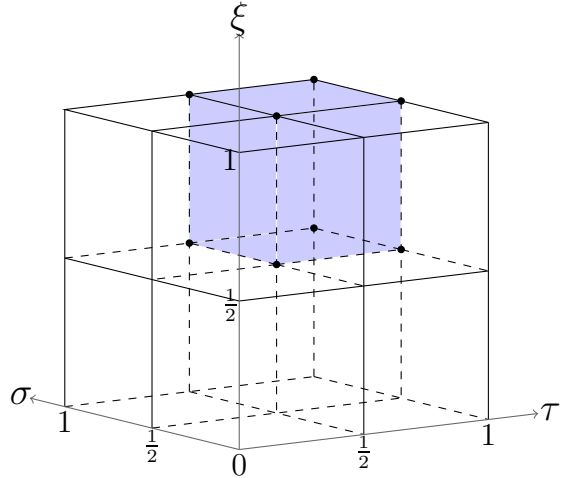
\begin{center}
\tdplotsetmaincoords{80}{-35}
\begin{tikzpicture}[tdplot_main_coords, scale=4.0,]
   \coordinate(A) at (0.5,0.5,0.5);
    \coordinate(B) at (1,0.5,0.5);
    \coordinate(C) at (1,0.5,1);
    \coordinate(D) at (0.5,0.5,1);
    \coordinate(E) at (0.5,1,0.5);
    \coordinate(F) at (1,1,0.5);
    \coordinate(G) at (1,1,1);
    \coordinate(H) at (0.5,1,1);
    \filldraw[black!10, fill=blue!20](E)--(F)--(G)--(H);
    \filldraw[black!10, fill=blue!20](A)--(B)--(F)--(E);
    \filldraw[black!10, fill=blue!20](C)--(G)--(H)--(D);
    \filldraw[black!0, fill=blue!20](A)--(B)--(C)--(D)--(A);
    \draw [dashed] (A)--(B);
    \draw (C)--(D);
    \draw [dashed] (D)--(0.5,0.5,0);
    \draw [dashed] (C)--(1,0.5,0);
    \draw [dashed] (E)--(F);
    \draw [dashed] (0.5,1,0)--(H);
    \draw (H)--(G);
    \draw [dashed] (1,1,0)--(1,0,0);
    \draw [dashed] (0.5,1,0)--(0.5,0,0);
    \draw [dashed] (0,0.5,0)--(1,0.5,0);
    \draw [dashed] (1,1,0)--(G);
    \draw [dashed] (A)--(E);
    \draw [dashed] (B)--(F);
    \draw (C)--(G);
    \draw (D)--(H);
    \draw (D)--(0,0.5,1)--(0,1,1)--(H);
    \draw (0,0.5,0)--(0,0.5,1);
    \draw (0,1,0)--(0,1,1);
    \draw (0,0.5,0.5)--(0,1,0.5);
    \draw [dashed] (0,1,0)--(1,1,0);
    \draw [dashed] (0,1,0.5)--(E);
    \draw [dashed] (A)--(0,0.5,0.5);
    \draw (0,0.5,0.5)--(0,0,0.5);
    \draw [dashed] (0.5,0.5,0.5)--(0.5,0,0.5);
    \draw [dashed] (1,0.5,0.5)--(1,0,0.5);
    \draw (1,0.5,1)--(1,0,1);
    \draw (0.5,0.5,1)--(0.5,0,1);
    \draw (0,0.5,1)--(0,0,1);
    \draw (0,0,0.5)--(1,0,0.5);
    \draw (0,0,1)--(1,0,1);
    \draw (0.5,0,0)--(0.5,0,1);
    \draw (1,0,0)--(1,0,1);
    \draw[->, black!60] (0, 0,0) -- (1.2, 0,0);
    \draw[->, black!60] (0, 0,0) -- (0, 1.2,0);
    \draw[->, black!60] (0, 0,0) -- (0, 0,1.4);
    \draw node at (1.25, 0,0) {\Large $\tau$};
    \draw node at (0, 1.25,0) {\Large $\sigma$};
    \draw node at (0, 0,1.45) {\Large $\xi$};
    \draw node at (0, 0,-0.05) {\large $0$};
    \draw node at (1, 0,-0.05) {\large $1$};
    \draw node at (0, 1,-0.05) {\large $1$};
    \draw node at (0.5, 0,-0.06) { $\frac{1}{2}$};
    \draw node at (0, 0.5,-0.06) { $\frac{1}{2}$};
    \draw node at (-0.07, -0.05,0.5) { $\frac{1}{2}$};
    \draw node at (-0.1, -0.09,1) {\large $1$};
    \draw[fill=black](0.5,0.5,0.5) circle (0.3pt);
    \draw[fill=black](1,1,1) circle (0.3pt);
    \draw[fill=black](0.5,1,1) circle (0.3pt);
    \draw[fill=black](0.5,0.5,1) circle (0.3pt);
    \draw[fill=black](1,0.5,1) circle (0.3pt);
    \draw[fill=black](1,0.5,0.5) circle (0.3pt);
    \draw[fill=black](0.5,1,0.5) circle (0.3pt);
    \draw[fill=black](1,1,0.5) circle (0.3pt);
    
\end{tikzpicture}
\end{center}
\begin{center}
{\bf FIG. 01:} Region $R_{A1}$ of Analyticity de $S(t)=e^{\mathcal{B}_1t}$
\end{center}

\begin{proof}
We will prove this theorem will be proved using the Theorem \ref{LiuZAnalyticity},  so the two conditions  \eqref{EixoImaginary} and \eqref{Analiticity} must be proved.
From  Lemma \ref{EixoImaginary01},  the condition \eqref{EixoImaginary} is verified,  it remains to verify the condition    \eqref{Analiticity},  note that proving this condition is equivalent to:  let $\delta>0$ there exists a constant $C_\delta > 0$ such that the solutions of \eqref{Eq05A}--\eqref{Eq08A}
for $|\lambda|\geq  \delta$  satisfy the inequality
 \begin{equation}\label{EquivAnaliticity}
 |\lambda|\|U\|^2_\mathcal{H}\leq C_\delta\|F\|_\mathcal{H}\|U\|_\mathcal{H}.
 \end{equation}
 It is not difficult to see that this inequality \eqref{EquivAnaliticity} arises from the inequalities of the Lemmas \ref{Regularidad} and \ref{Lema001Analiticity},    so the proof of this theorem is finished.
\end{proof}
\subsection{Regularity of the second system}
Next,  two fundamental Lemmas for the determination of the Gevrey class and the proof of the analyticity of the semigroup $S_2(t)=e^{\mathcal{B}_2t}$ are presented.
\begin{lemma}\label{LemaExponencial2}
Let $S_2(t)=e^{\mathcal{B}_2t}$  be a $C_0$-semigroup on contractions on Hilbert space $\mathcal{H}=\mathcal{H}_2$,  the solutions of the system \eqref{Eq01B2}--\eqref{Eq09A} satisfy the inequality
\begin{equation}\label{EstimaExponential2} 
\limsup_{|\lambda|\to\infty}\|(i\lambda I-\mathcal{B}_2) ^{-1}\|_{\mathcal{L}(\mathcal{H})}<\infty.
\end{equation}
\end{lemma}
\begin{proof} 
To show the \eqref{EstimaExponential} inequality,  it suffices to show that,  given $\delta>0$ there exists a constant $C_\delta>0$ such that  the solutions of the system  \eqref{Eq01B2}--\eqref{Eq09A} for $|\lambda|>\delta$ satisfy the inequality
\begin{equation}\label{EstimaEquivExp2}
\dfrac{\|U\|_{\mathcal{H}}}{\|F\|_{\mathcal{H}}}\leq C_\delta\Longleftrightarrow \|U\|^2_{\mathcal{H}}=\|(i\lambda I-\mathcal{B}_2)^{-1} F\|^2_{\mathcal{H}} \leq C_\delta \|F\|_\mathcal{H}\|U\|_{\mathcal{H}}.
\end{equation}
If  $F=(f^1,f^2,f^3,f^4, f^5, f^6)\in \mathcal{H}$ then the solution $U=(\phi,\Phi,\psi, \Psi,\theta,\Theta)\in\hbox{D}(\mathcal{B}_2)$ the resolvent equation $(i\lambda I- \mathcal{B}_2)U=F$ can be written in the form
\begin{eqnarray}
i\lambda \phi-\Phi &=& f^1\quad\rm{in}\quad D(A^\frac{1}{2}),\label{esp-10B}\\
i\lambda \Phi-\frac{\kappa}{\rho_1}(\phi_x+\psi)_x+\dfrac{\mu}{\rho_1}\Theta_x+\frac{\mu_1}{\rho_1}A^\tau \Phi&=& f^2\quad\rm{in}\quad D(A^0),\label{esp-20B}\\
i\lambda  \psi-\Psi &=&f^3\quad \rm{in}\quad D(A^\frac{1}{2}),\label{esp-30B}\\
i\lambda \Psi+\dfrac{b}{\rho_2}A\psi+\dfrac{\kappa}{\rho_2}(\phi_x+\psi)-\dfrac{\mu}{\rho_2}\Theta+\dfrac{\mu_2}{\rho_2}A^\sigma \Psi &=& f^4 \quad\rm{in}\quad D(A^0),\label{esp-40B}\\
i\lambda\theta-\Theta&=&f^5\quad{ \rm in}\quad D(A^\frac{1}{2}),\label{esp-50B}\\
i\lambda\Theta+\dfrac{\delta}{\rho_3}A\theta+\dfrac{\mu}{\rho_3}(\Phi_x+\Psi)+\dfrac{\gamma}{\rho_3}A^\xi\Theta&=& f^6\quad {\rm in}\quad D(A^0). \label{esp-60B}
\end{eqnarray}
Using the fact that the operator is dissipative $ \mathcal{B}_2$,  result in
\begin{multline}\label{dis-10B}
\mu_1\|A^\frac{\tau}{2}\Phi\|^2+\mu_2\|A^\frac{\sigma}{2} \Psi\|^2+\gamma\|A^\frac{\xi}{2}\Theta\|^2= \text{Re}\dual{(i\lambda -\mathcal{B}_2)U}{U}= \text{Re}\dual{F}{U}\leq \|F\|_\mathcal{H}\|U\|_\mathcal{H}.
\end{multline}
On the other hand,  performing the duality product of \eqref{esp-20B} for $\rho_1 \phi$,   and as
the operators $A^\nu$ for all $\nu\in\mathbb{R}$ are self-adjoint,  lead to
\begin{multline*}
\kappa\dual{(\phi_x+\psi)}{\phi_x}= \rho_1\|\Phi\|^2+\rho_1\dual{\Phi}{f^1}
 -i\lambda\mu_1\|A^\frac{\tau}{2}\phi\|^2
+\mu_1\dual{A^\frac{\tau}{2} f^1}{A^\frac{\tau}{2} \phi}
+\rho_1\dual{f^2}{\phi}+\mu\dual{\Theta}{\phi_x},
\end{multline*}
now performing the duality product of \eqref{esp-40B} for $\rho_2 \psi$,   and as
the operators $A^\nu$ for all $\nu\in\mathbb{R}$ are self-adjoint,  result in 
\begin{multline*}
\kappa\dual{(\phi_x+\psi)}{\psi}+b\|A^\frac{1}{2}\psi\|^2=\rho_2\|\Psi\|^2
+\rho_2\dual{\Psi}{f^3}\\
+\mu\dual{\Theta}{\psi}-i\lambda\mu_2\|A^\frac{\sigma}{2}\psi\|^2+\mu_2\dual{A^\frac{\sigma}{2}f^3}{A^\frac{\sigma}{2}\psi}+\rho_2\dual{f^4}{\psi}.
\end{multline*}
Adding the last 2 equations,  result in 
\begin{multline}\label{Eq001ExponentialB}
\kappa\|\phi_x+\psi\|^2+b\|A^\frac{1}{2}\psi\|^2=
\rho_1\|\Phi\|^2+\rho_2\|\Psi\|^2
-i\lambda \{ \mu_1 \|A^\frac{\tau}{2}u\|^2
+\mu_2\|A^\frac{\sigma}{2}\psi\|^2 \}\\
 +\rho_1\dual{\Phi}{f^1}+\rho_2\dual{\Psi}{f^3}
+\mu_1\dual{A^\frac{\tau}{2}f^1}{A^\frac{\tau}{2}\phi} +\mu_2\dual{A^\frac{\sigma}{2}f^3}{A^\frac{\sigma}{2}\phi}\\
+\rho_1\dual{f^2}{\phi}+\rho_2\dual{f^4}{\psi}+\mu\dual{\Theta}{\phi_x+\psi}.
\end{multline}
Taking real part, applying Cauchy-Schwarz and Young inequalities,  for $\varepsilon>0$ exists $C_\varepsilon>0$ such that
\begin{equation*}
\kappa\|\phi_x+\psi\|^2+b\|A^\frac{1}{2}\psi\|^2 
\leq \rho_1\|\Phi\|^2+\rho_2\|\Psi\|^2+\varepsilon\|\phi_x+\psi\|^2+ C_\varepsilon\|\Theta\|^2+C_\delta\|F\|_\mathcal{H}\|U\|_\mathcal{H}.
\end{equation*}
From  estimative \eqref{dis-10B}  and the fact $0\leq\frac{\tau}{2},\;0\leq\frac{\sigma}{2}$ and $0\leq\frac{\xi}{2}$  the continuous embedding $D(A^{\theta_2}) \hookrightarrow D(A^{\theta_1}),\;\theta_2>\theta_1$,  result in 
\begin{eqnarray}\label{Exponential002B}
\kappa\|\phi_x+\psi\|^2+b\|A^\frac{1}{2}\psi\|^2&\leq& C_\delta\|F\|_\mathcal{H}\|U\|_\mathcal{H}.
\end{eqnarray}
On the other hand,  performing the duality product of \eqref{esp-60B} for $\rho_3 \theta$,    and remembering that the operators $A^\nu$ for all $\nu\in\mathbb{R}$ are
self-adjoint and from \eqref{esp-50B},  lead to
\begin{multline*}
\delta\|A^\frac{1}{2}\theta\|^2=\rho_3\|\Theta\|^2+\rho_3\dual{\Theta}{f^5}+\mu\dual{\Phi}{\theta_x}-\mu\dual{\Psi}{\theta}
-i\lambda\gamma\|A^\frac{\xi}{2}\theta\|^2+\gamma\dual{A^\frac{\xi}{2}f^5}{A^\frac{\xi}{2}\theta}+\rho_3\dual{f^6}{\theta}.
\end{multline*}
Taking the real part,  using the inequalities Cauchy-Schwarz and Young,   for $\varepsilon>0$ exists $C_\varepsilon>0$ which does not depend on $\lambda$ such that
\begin{multline*}
\delta\|A^\frac{1}{2}\theta\|^2
\leq C\{\|\Theta\|^2+\|\Theta\|\|f^5\|+\|A^\frac{\xi}{2}f^5\|\|A^\frac{\xi}{2}\theta\|+\|f^6\|\|\theta\|\}\\+C_\varepsilon\{\|\Phi\|^2+\|\Psi\|^2\}+\varepsilon\{ \|\theta\|^2+\|A^\frac{1}{2}\theta\|^2\}.
\end{multline*}
Then, from estimative \eqref{dis-10B}  and the fact $0\leq\frac{\tau}{2},\; 0\leq\frac{\sigma}{2}$ and $0\leq\frac{\xi}{2}\leq\frac{1}{2}$  the continuous embedding $D(A^{\theta_2}) \hookrightarrow D(A^{\theta_1}),\;\theta_2>\theta_1$,   result in 
\begin{equation}\label{Exponential003B}
\beta\kappa\delta\|A^\frac{1}{2}\theta\|^2\leq C_\delta\|F\|_\mathcal{H}\|U\|_\mathcal{H}.
\end{equation}
Finally,  from estimates \eqref{dis-10B}, \eqref{Exponential002B} and \eqref{Exponential003B},  the proof of this lemma is finished.
\end{proof}
\begin{lemma}\label{RegularidadB}
Let $\delta> 0$. There exists a constant $C_\delta > 0$ such that the solutions of \eqref{Eq01B2}--\eqref{Eq09A} for $|\lambda|\geq \delta$ satisfy the inequalities
\begin{eqnarray}\label{EqPrincipalLemaB}
\hspace*{-1.0cm}(i)\; |\lambda|[\kappa\|\phi_x+\psi\|^2+b\|A^\frac{1}{2}\psi\|^2] &\leq
&  |\lambda|[\rho_1\|\Phi\|^2+\rho_2\|\Psi\|^2+C\|\Theta\|^2]  +C_\delta\|F\|_\mathcal{H}\|U\|_\mathcal{H}.\\
\label{EqPrincipalLemaAB}
\hspace*{-1.0cm}(ii)\; \delta|\lambda|\|A^\frac{1}{2}\theta\|^2 &\leq & C |\lambda|\{ \|\Theta\|^2+\|\Phi\|^2\}+ C_\delta\|F\|_\mathcal{H}\|U\|_\mathcal{H}.
\end{eqnarray}
\end{lemma}
\begin{proof}
{\bf Item $(i)$:} Performing the duality product of \eqref{esp-20B} for $\rho_1\lambda \phi$,   and remembering that the operators $A^\nu$ for all $\nu\in \mathbb{R} $ are self-adjunct,  lead to 
\begin{eqnarray*}
\kappa\lambda\dual{(\phi_x+\psi)}{\phi_x}& =&\rho_1\lambda\|\Phi\|^2+\rho_1\dual{\lambda \Phi}{f^1}  -i\mu_1\|A^\frac{\tau}{2}\Phi\|^2\\
& & -i\mu_1\beta\gamma\dual{A^\tau \Phi}{f^1}
+i\rho_1\dual{f^2}{\Phi}+i\rho_1\dual{f^2}{f^1}+\mu\dual{\Theta}{\lambda\phi_x}.
\end{eqnarray*}
Now performing the duality product of \eqref{esp-40B} for $\rho_2\lambda \psi$,  and remembering that the operators $A^\nu$ for all $\nu\in \mathbb{R} $ are
self-adjunct, resulting in  
\begin{eqnarray*}
\lambda[\kappa\dual{(\phi_x+\psi)}{\psi}+
b\|A^\frac{1}{2}\psi\|^2]\hspace*{-0.3cm}
& = &\hspace*{-0.3cm}\rho_2\lambda\|\Psi\|^2+\rho_2\dual{\lambda \Psi}{f^3}- i\mu_2\|A^\frac{\sigma}{2}\Psi\|^2  -i\mu_2\dual{A^\sigma \Psi}{f^3} \\
& & +i\rho_2\dual{f^4}{\Psi} +i\rho_2\dual{f^4}{f^3}+\mu\dual{\Theta}{\lambda\psi}.
\end{eqnarray*}
Adding the last 2 equations,  result in 
\begin{multline}\label{Eq001GevreyB}
\lambda[\kappa\|\phi_x+\psi\|^2+b\|A^\frac{1}{2}\psi\|^2] =
\lambda[\rho_1\|\Phi\|^2+\rho_2\|\Psi\|^2]
-i\{\mu_1\|A^\frac{\tau}{2}\Phi\|^2\\+\mu_2\|A^\frac{\sigma}{2}\Psi\|^2\}+\rho_1\dual{\lambda \Phi}{f^1}+\rho_2\dual{\lambda \Psi}{f^3}-i\{\mu_1 \dual{A^\tau \Phi}{f^1}\\
+\mu_2\dual{A^\sigma \Psi}{f^3} \}
+i\rho_1\dual{f^2}{\Phi} +i\rho_2\dual{f^4}{\Psi}
+i\rho_1\dual{f^2}{f^1}\\
+i\rho_2\dual{f^4}{f^3}+\mu\dual{\sqrt{|\lambda|}\Theta}{\dfrac{\lambda}{\sqrt{|\lambda|}}(\phi_x+\psi)}.   
\end{multline}
On the other hand, from \eqref{esp-20B}  and \eqref{esp-40B},  result in 
\begin{multline}\label{Eq002GevreyB}
\rho_1\dual{\lambda \Phi}{f^1}+\rho_2\dual{\lambda \Psi}{f^3}
 =i\{ \kappa\dual{(\phi_x+\psi)}{f^1_x}-\mu\dual{\Theta}{f^1_x}\\
 +\mu_1\dual{A^\tau \Phi}{f^1}-\rho_1\dual{f^2}{f^1}
 +b\dual{A^\frac{1}{2}\psi}{A^\frac{1}{2}f^3}+\kappa\dual{(\phi_x+\psi)}{f^3}\\
 -\mu\dual{\Theta}{f^3}+\mu_2\dual{A^\sigma \Psi}{f^3}-\rho_2\dual{f^4}{f^3}\}.
\end{multline}
Using the identity \eqref{Eq002GevreyB} in the \eqref{Eq001GevreyB} equation and simplifying,  leads to
\begin{eqnarray}
\label{Eq003GevreyB}
\hspace*{-0.45cm}\lambda[\kappa\|\phi_x+\psi\|^2+b\|A^\frac{1}{2}\psi\|^2] \hspace*{-0.25cm} &= & \hspace*{-0.25cm}
\lambda[\rho_1\|\Phi\|^2+\rho_2\|\Psi\|^2]
-i\{\mu_1\|A^\frac{\tau}{2}\Phi\|^2
+\mu_2\|A^\frac{\sigma}{2}\Psi\|^2\}+i\kappa\dual{\phi_x}{f_x^1}\\
\nonumber
& & 
\hspace*{-0.25cm}+i\kappa\dual{\psi}{f_x^1}
-i\mu\dual{\Theta}{f^1_x}
+ib \dual{A^\frac{1}{2}\psi}{A^\frac{1}{2}f^3}+i\kappa\dual{\phi_x}{f^3}+i\kappa\dual{\psi}{f^3}\\
\nonumber
& &\hspace*{-0.25cm} 
-i\mu\dual{\Theta}{f^3}+i\rho_1\dual{f^2}{\Phi}+i\rho_2\dual{f^4}{\Psi} +\mu\dual{\sqrt{|\lambda|}\Theta}{\dfrac{\lambda}{\sqrt{|\lambda|}}(\phi_x+\psi)}.  
\end{eqnarray}
As for $\varepsilon>0$,  exists $C_\varepsilon>0$,  such that  
$\bigg|\mu\dual{\sqrt{|\lambda|}\Theta}{\dfrac{\lambda}{\sqrt{|\lambda|}}(\phi_x+\psi)}\bigg|\leq C_\varepsilon|\lambda|\|\Theta\|^2+\varepsilon|\lambda|\|\phi_x+\psi\|^2,$
taking the real part of the equation \eqref{Eq003GevreyB},  applying the Cauchy-Schwarz and Young inequalities and    estimative \eqref{EstimaEquivExp2}  of Lemma \ref{LemaExponencial2},  the proof of item $(i)$  this lemma is completed.\\
On the other hand,  performing the duality product of \eqref{esp-60B} for $\rho_3\lambda \theta$,   and remembering that the operators $A^\nu$ for all $\nu\in \mathbb{R } $ are self-adjunct,  result in 
\begin{equation}
\delta\lambda\|A^\frac{1}{2}\theta\|^2=\rho_3\lambda\| \Theta\|^2+i\mu\dual{\Phi}{\Theta_x}-i\gamma\|A^\frac{\xi}{2}\Theta\|^2+i\rho_3\dual{f^6}{\Theta}-i\mu\dual{\Psi}{\Theta}.
\label{Eq000LthetaB}
\end{equation}
As
\begin{equation}
\label{Eq001LthetaB}
i\mu\dual{\Phi}{\Theta_x}=i\mu\dual{\Phi}{i\lambda\theta_x-f_x^5}=\mu\dual{\sqrt{|\lambda|}\Phi}{\dfrac{\lambda}{\sqrt{|\lambda|}}\theta_x}
-i\mu\dual{\Phi}{f_x^5}.
\end{equation}
Then,  using \eqref{Eq001LthetaB}  in \eqref{Eq000LthetaB} and taking real part,  lead to 
\begin{multline}\label{Eq002LthetaB}
\delta|\lambda|\|A^\frac{1}{2}\theta\|^2=\rho_3\lambda\| \Theta\|^2+\mu\dual{\sqrt{|\lambda|}\Phi}{\dfrac{\lambda}{\sqrt{|\lambda|}}\theta_x}
-i\mu\dual{\Phi}{f_x^5}
+i\rho_3\dual{f^6}{\Theta}-i\mu\dual{\Psi}{\Theta},
\end{multline}
applying Cauchy-Schwarz and Young inequalities,  for $\varepsilon>0$ exists $C_\varepsilon>0$,  such that 
\begin{equation*}
\delta|\lambda|\|A^\frac{1}{2}\theta\|^2 \leq \rho_3|\lambda|\|\Theta\|^2+ \varepsilon|\lambda|\|A^\frac{1}{2}\theta\|^2+ C_\varepsilon|\lambda|\|\Phi\|^2 
+C\{\|\Phi\|\|A^\frac{1}{2}f^5\|+\|f^6\|\|\Theta\|+\|\Phi\|^2+\|\Theta\|^2 \}.
\end{equation*}
Then,  from Lemma \ref{LemaExponencial2},   the proof of item $(ii)$ of this lemma is finished.
\end{proof}

\subsubsection{Gevrey  class  of the second system}\label{3.2.1}
\begin{theorem} \label{GevreyLaminadoB} The  semigroup  $S(t)=e^{\mathcal{B}_2t}$  associated to system  \eqref{Eq01B2}--\eqref{Eq09A} is of Gevrey class $s>\dfrac{1+r}{2r}$ for $r=\min\{\tau,\sigma,\xi\}$,  for all  $(\tau,\sigma,\xi)\in R_{CG}:=(0,1)^3$ and $\frac{2r}{r+1}\in (0,1)$.
\end{theorem}
\begin{proof}
From the resolvent  equation $F=(i\lambda I-\mathcal{B}_2)U$  for $\lambda\in\mathbb{R}$,   result in 
$U=(i\lambda I-\mathcal{B}_2)^{-1}F$.  Furthermore to show  \eqref{Eq1.5Tebon2020} this is   theorem\eqref{Theorem1.2Tebon} it is enough to show:
\begin{equation}\label{EstimaEquivalenteGevreyB}
|\lambda |^\frac{2r}{r+1}\|U\|_\mathcal{H}^2\leq C_\delta \|F\|_\mathcal{H}\|U\|_\mathcal{H}\quad\rm{for}\quad \frac{2r}{r+1} \in (0,1).
\end{equation}
where  $r=\min\{\tau,\sigma,\xi\}$,  for all  $(\tau,\sigma,\xi)\in (0,1)^3$.
\\
Next,  $|\lambda|^\frac{2\tau}{1+\tau}\|\Phi\|$,  $|\lambda|^\frac{2\sigma}{1+\sigma}\|\Psi\|$ and $|\lambda|^\frac{2\xi}{1+\xi}\|\Theta\|$ will be estimated.  \\
{\bf Let's start by estimating the term $|\lambda|^\frac{2\tau}{1+\tau}\|\Phi\|$:}  It is  assumed that   $|\lambda|>1$.  Set $\Phi=\Phi_1+\Phi_2$, where $\Phi_1\in D(A)$ and $\Phi_2\in D(A^0)$, with 
\begin{multline}\label{Eq110AnalyRRB}
i\lambda \Phi_1+A \Phi_1=f^2, 
\hspace{1cm} i\lambda \Phi_2=-\dfrac{\kappa}{\rho_1}A\phi+\dfrac{\kappa}{\rho_1}\psi_x -\dfrac{\mu}{\rho_1}\Theta_x-\dfrac{\mu_1}{\rho_1}A^\tau \Phi+A\Phi_1.
\end{multline} 
Firstly,  applying the product duality  on the first equation in \eqref{Eq110AnalyRRB} by $\Phi_1$,  then by $A^{\Phi_1}$ and recalling
that the operator $A$ is self-adjoint, resulting in
self-adjoint, resulting in 
\begin{equation}\label{Eq112AnalyRRB}
|\lambda|\|\Phi_1\| +|\lambda|^\frac{1}{2}\|A^\frac{1}{2}\Phi_1\|+\|A\Phi_1\|\leq C\|F\|_\mathcal{H}.
\end{equation}
Applying the $A^{-\frac{1}{2}}$ operator on the second equation of \eqref{Eq110AnalyRRB},  lead to
\begin{equation*}
i\lambda A^{-\frac{1}{2}}\Phi_2= -\dfrac{\kappa}{\rho_1}A^\frac{1}{2}\phi+\dfrac{\kappa}{\rho_1}A^{-\frac{1}{2}}\psi_x-\dfrac{\mu}{\rho_1}A^{-\frac{1}{2}}\Theta_x-\dfrac{\mu_1}{\rho_1}A^{\tau-\frac{1}{2}} \Phi+A^\frac{1}{2}\Phi_1,
\end{equation*}
 then,  as $\|A^{-\frac{1}{2}}\psi_x\|^2=\|\psi\|^2$  and  $\tau-\frac{1}{2}\leq \frac{\tau}{2}$ taking into account the continuous embedding $D(A^{\theta_2}) \hookrightarrow D(A^{\theta_1}),\;\theta_2>\theta_1$,  result in 
\begin{equation}\label{Eq113AAnalyB}
|\lambda|^2\|A^{-\frac{1}{2}} \Phi_2\|^2 \leq C\{\|A^\frac{1}{2}\phi\|^2+ \|A^\frac{\tau}{2}\Phi\|^2+\|\psi\|^2+\|\Theta\|^2\}+\|A^\frac{1}{2}\Phi_1\|^2
\end{equation}
Using  \eqref{EstimaEquivExp2} and  estimative  \eqref{Eq112AnalyRRB},  result in 
\begin{equation}\label{Eq113AnalyRRB}
\|A^{-\frac{1}{2}}\Phi_2\|^2\leq C|\lambda|^{-\frac{4\tau+2}{\tau+1}}\{ |\lambda|^\frac{2\tau}{\tau+1}\|F\|_\mathcal{H}\|U\|_\mathcal{H}+\|F\|^2_\mathcal{H}\}\quad {\rm for}\quad 0\leq\tau\leq 1.
\end{equation}
On the  other hand, from $\Phi_2=\Phi-\Phi_1$,  \eqref{dis-10B} and  as $\frac{\tau}{2}\leq\frac{1}{2}$ the second inequality of \eqref{Eq112AnalyRRB},   result in 
\begin{eqnarray}\label{Eq114AnalyRRB}
\|A^\frac{\tau}{2} \Phi_2\|^2& \leq &C\{ \|A^\frac{\tau}{2} \Phi\|^2+\|A^\frac{\tau}{2}\Phi_1\|^2\}
\leq  C|\lambda|^{-\frac{2\tau}{\tau+1}}\{|\lambda|^\frac{2\tau}{\tau+1}\|F\|_\mathcal{H}\|U\|_\mathcal{H}+\|F\|^2_\mathcal{H}\}.
\end{eqnarray}
Now,  by Lions' interpolations inequality $0\in[-\frac{1}{2},\frac{\tau}{2}]$, result in 
\begin{equation}\label{Eq115AnalyRRB} 
\|\Phi_2\|^2\leq C (\|A^{-\frac{1}{2}}\Phi_2\|^2)^\frac{\tau}{1+\tau}(\|A^\frac{\tau}{2}\Phi_2\|^2)^\frac{1}{1+\tau}.
\end{equation}
Then, using \eqref{Eq113AnalyRRB} and \eqref{Eq114AnalyRRB} in \eqref{Eq115AnalyRRB}, result in 
\begin{equation}\label{Eq118AnalyRRB}
 \|\Phi_2\|^2\leq C|\lambda|^{\frac{-4\tau}{(1+\tau)}}\{ |\lambda|^\frac{2\tau}{\tau+1} \|F\|_\mathcal{H}\|U\|_\mathcal{H}+\|F\|^2_\mathcal{H}\}.
\end{equation}
Therefore,   as $\|\Phi\|^2\leq  \|\Phi_1\|^2+ \|\Phi_2\|^2$ from first inequality of  \eqref{Eq112AnalyRRB},  \eqref{Eq118AnalyRRB}   and $|\lambda|^{-2}\leq |\lambda|^\frac{-4\tau}{1+\tau}$, we have
\begin{multline}\label{Eq119AnalyRRB}
|\lambda|^\frac{2\tau}{1+\tau}\|\Phi\|\leq C_\delta\|F\|_\mathcal{H}\\
\Longleftrightarrow \quad |\lambda|\|\Phi\|^2\leq C_\delta|\lambda|^\frac{1-3\tau}{1+\tau}\{|\lambda|^\frac{2\tau}{\tau+1} \|F\|_\mathcal{H}\|U\|_\mathcal{H}+\|F\|^2_\mathcal{H}\}\quad\rm{for}\quad 0\leq\tau\leq 1.
\end{multline}

{\bf On the other hand,   let's now estimate the missing term  $|\lambda|^\frac{2\sigma}{1+\sigma}\|\Psi\|$:}  Set $\Psi=\Psi_1+\Psi_2$, where $\Psi_1\in D(A)$ and $\Psi_2\in D(A^0)$, with 
\begin{equation}\label{Eq110AnalyRRWB}
\hspace*{-0.2cm}i\lambda \Psi_1+A \Psi_1=f^4 
\quad{\rm and} \quad i\lambda \Psi_2=-\dfrac{b}{\rho_2}A\psi -\dfrac{\kappa}{\rho_2}\phi_x-\dfrac{\kappa}{\rho_2}\psi+\dfrac{\mu}{\rho_2}\Theta-\dfrac{\mu_2}{\rho_2}A^\sigma \Psi+A\Psi_1.
\end{equation} 
Firstly,  applying the product duality  on the first equation in \eqref{Eq110AnalyRRWB} by $\Psi_1$,  then by $A\Psi_1$ and  recalling that the operator $A$  is self-adjoint, resulting in 
\begin{equation}\label{Eq112AnalyRRWB}
|\lambda|\|\Psi_1\| +|\lambda|^\frac{1}{2}\|A^\frac{1}{2}\Psi_1\|+\|A\Psi_1\|\leq C \|F\|_\mathcal{H}.
\end{equation}
It arises  from the second equation in \eqref{Eq110AnalyRRWB}  that
\begin{equation*}
i\lambda A^{-\frac{1}{2}}\Psi_2= -\dfrac{b}{\rho_2}A^\frac{1}{2}\psi-\dfrac{\kappa}{\rho_2}A^{-\frac{1}{2}}\phi_x-\dfrac{\kappa}{\rho_2}A^{-\frac{1}{2}}\psi+\dfrac{\mu}{\rho_2}A^{-\frac{1}{2}}\Theta-\dfrac{\mu_2}{\rho_2}A^{\sigma-\frac{1}{2}}\Psi+A^\frac{1}{2}\Psi_1,
\end{equation*}
 then,  as $\|A^{-\frac{1}{2}}\phi_x\|^2=\|\phi\|^2$,  $\|A^{-\frac{1}{2}}\Theta_x\|^2=\|\Theta\|^2 $ and  $\sigma-\frac{1}{2}\leq \frac{\sigma}{2}$ taking into account the continuous embedding $D(A^{\theta_2}) \hookrightarrow D(A^{\theta_1}),\;\theta_2>\theta_1$,   result in 
\begin{equation}\label{Eq113AAnalyRRWB}
|\lambda|^2\|A^{-\frac{1}{2}} \Psi_2\|^2\leq C\{ \|\phi\|^2+\|A^{\frac{1}{2}}\psi\|^2+\|A^{-\frac{1}{2}}\Theta\|^2+\|A^\frac{\sigma}{2}\Psi\|^2\}+\|A^\frac{1}{2}\Psi_1\|^2
\end{equation}
Using estimative \eqref{EstimaEquivExp2} of Lemma \ref{LemaExponencial2}  and  second estimative of  \eqref{Eq112AnalyRRWB},   yields 
\begin{equation}\label{Eq113AnalyRRWB}
\|A^{-\frac{1}{2}}\Psi_2\|^2\leq C|\lambda|^{-\frac{4\sigma+2}{\sigma+1}}\{|\lambda|^\frac{2\sigma}{\sigma+1}\|F\|_\mathcal{H}\|U\|_\mathcal{H}+\|F\|_\mathcal{H}\}\quad\text{for}\quad 0\leq \sigma\leq 1.
\end{equation}
On the  other hand, from $\Psi_2=\Psi-\Psi_1$,  \eqref{dis-10B} and  as $0\leq\frac{\sigma}{2}\leq\frac{1}{2}$, the second inequality of \eqref{Eq112AnalyRRWB},  result in  
\begin{multline}\label{Eq114AnalyRRWB}
\|A^\frac{\sigma}{2} \Psi_2\|^2 \leq  C\{ \|A^\frac{\sigma}{2} \Psi\|^2+\|A^\frac{\sigma}{2}\Psi_1\|^2\}
\leq  C|\lambda|^{-\frac{2\sigma}{\sigma+1}}\{|\lambda|^\frac{2\sigma}{\sigma+1}\|F\|_\mathcal{H}\|U\|_\mathcal{H}+\|F\|^2_\mathcal{H}\}.
\end{multline}
Now using  interpolation inequality $0\in [-\frac{1}{2},\frac{\sigma}{2}]$.   Since
$0=\eta\bigg(-\dfrac{1}{2}\bigg)+(1-\eta)\dfrac{\sigma}{2},\quad {\rm for}\quad \eta=\dfrac{\sigma}{1+\sigma},$
 using \eqref{Eq113AnalyRRWB} and \eqref{Eq114AnalyRRWB} we get that
 \begin{eqnarray}
 \|\Psi_2\|^2 &\leq & C_\delta (\|A^{-\frac{1}{2}}\Psi_2\|^2)^{\eta} (\|A^\frac{\sigma}{2}\Psi_2\|^2)^{1-\eta} 
 \label{Eq118AnalyRRWB}
\leq C |\lambda|^{\frac{-4\sigma}{1+\sigma}}\{|\lambda|^\frac{2\sigma}{\sigma+1}\|F\|_\mathcal{H}\|U\|_\mathcal{H}+\|F\|^2_\mathcal{H}\}.
\end{eqnarray}
Also,   as $\|\Psi\|^2\leq  C\{\|\Psi_1\|^2+  \|\Psi_2\|^2\}$ from first inequality of  \eqref{Eq112AnalyRRWB}, estimative \eqref{Eq118AnalyRRWB} and as $|\lambda|^{-2}\leq |\lambda|^\frac{-2\sigma}{1+\sigma}$,   result in 
\begin{multline}\label{Eq119AnalyRRWB}
|\lambda|^\frac{2\sigma}{1+\sigma}\|\Psi\|\leq C_\delta\|F\|_\mathcal{H}\\
\Longleftrightarrow\; |\lambda|\|\Psi\|^2 \leq C_\delta|\lambda|^\frac{1-3\sigma}{\sigma+1}\{|\lambda|^\frac{2\sigma}{\sigma+1} \|F\|_\mathcal{H}\|U\|_\mathcal{H}+\|F\|^2_\mathcal{H}\}\quad\rm{for}\quad 0\leq\sigma\leq 1.
\end{multline}
{\bf Finally,  let's now estimate the missing term  $|\lambda|^\frac{2\xi}{1+\xi}\|\Theta\|$.}  Set $\Theta=\Theta_1+\Theta_2$,  where $\Theta_1\in D(A)$ and $\Theta_2\in D(A^0)$, with 
\begin{multline}\label{Eq110AnalyRRTB}
i\lambda \Theta_1+A \Theta_1=f^6 
\quad{\rm and}\quad
 i\lambda \Theta_2=-\dfrac{\delta}{\rho_3}A\theta -\dfrac{\mu}{\rho_3}\Phi_x-\dfrac{\mu}{\rho_3}\Psi-\dfrac{\gamma}{\rho_3}A^\xi \Theta+A\Theta_1.
\end{multline} 
Firstly,  applying the product duality  the first equation in \eqref{Eq110AnalyRRTB} by $\Theta_1$, then by $A\Theta_1$ and as the operators $A^\nu$ for all $\nu\in\mathbb{R}$ are self-adjoint,  result in
\begin{equation}\label{Eq112AnalyRRTB}
|\lambda|^2\|\Theta_1\|^2+|\lambda|\|A^\frac{1}{2}\Theta_1\|^2+\|A\Theta_1\|^2\leq C  \|F\|^2_\mathcal{H}.
\end{equation}
It arises  from the second equation in \eqref{Eq110AnalyRRTB}  that
\begin{equation*}
i\lambda A^{-\frac{1}{2}}\Theta_2= -\dfrac{\delta}{\rho_3}A^\frac{1}{2}\theta -\dfrac{\mu}{\rho_3}A^{-\frac{1}{2}}\Phi_x-\dfrac{\mu}{\rho_3}A^{-\frac{1}{2}}\Psi-\dfrac{\gamma}{\rho_3}A^{\xi-\frac{1}{2}} \Theta+A^\frac{1}{2}\Theta_1,
\end{equation*}
 then,  as  $\|A^{-\frac{1}{2}}\Phi_x\|^2=\|\Phi\|^2 $, $-\frac{1}{2}\leq 0$ and  $\xi-\frac{1}{2}\leq \frac{\xi}{2}$ taking into account the continuous embedding $D(A^{\theta_2}) \hookrightarrow D(A^{\theta_1}),\;\theta_2>\theta_1$ and using estimative \eqref{EstimaEquivExp2} and  estimative  \eqref{Eq112AnalyRRTB},    lead to
\begin{equation}\label{Eq113AnalyRRTB}
\|A^{-\frac{1}{2}}\Theta_2\|^2\leq C |\lambda|^{-\frac{4\xi+2}{\xi+1}}\{|\lambda|^\frac{2\xi}{\xi+1}\|F\|_\mathcal{H}\|U\|_\mathcal{H}+\|F\|_\mathcal{H}\}.
\end{equation}
On the  other hand, from $\Theta_2=\Theta-\Theta_1$,  \eqref{dis-10B} and  as $\frac{\xi}{2}\leq\frac{1}{2}$ the second inequality of \eqref{Eq112AnalyRRTB},  result in 
\begin{eqnarray}\label{Eq114AnalyRRTB}
\|A^\frac{\xi}{2} \Theta_2\|^2 & \leq &   C|\lambda|^{-\frac{2\xi}{\xi+1}}\{|\lambda|^\frac{2\xi}{\xi+1}\|F\|_\mathcal{H}\|U\|_\mathcal{H}+\|F\|^2_\mathcal{H}\}.
\end{eqnarray}
Now  using  interpolation inequality $0\in [-\frac{1}{2},\frac{\xi}{2}]$.   Since
$0=\eta\bigg(-\dfrac{1}{2}\bigg)+(1-\eta)\dfrac{\xi}{2},\quad {\rm for}\quad \eta=\dfrac{\xi}{1+\xi},$
 using \eqref{Eq113AnalyRRTB} and \eqref{Eq114AnalyRRTB} we get that
 \begin{eqnarray}
 \|\Theta_2\|^2 &\leq & C_\delta (\|A^{-\frac{1}{2}}\Theta_2\|^2)^{\eta} (\|A^\frac{\xi}{2}\Theta_2\|^2)^{1-\eta} 
 \label{Eq118AnalyRRTB}
\leq  |\lambda|^{\frac{-2\xi}{\xi+1}}\{\|F\|_\mathcal{H}\|U\|_\mathcal{H}+\|F\|^2_\mathcal{H}\}
\end{eqnarray}
Also as, $\|\Theta\|^2\leq C\{ \|\Theta_1\|^2+ \|\Theta_2\|^2\}$ from first inequality of  \eqref{Eq112AnalyRRTB},  \eqref{Eq118AnalyRRTB} and as $|\lambda|^{-2}\leq |\lambda|^\frac{-4\xi}{1+\xi}$,  result in 
\begin{multline}\label{Eq119AnalyRRTB}
|\lambda|^\frac{2\xi}{1+\xi}\|\Theta\|\leq C_\delta\|F\|_\mathcal{H}\\
\Longleftrightarrow\quad |\lambda|\|\Theta\|^2\leq C_\delta |\lambda|^\frac{1-3\xi}{1+\xi}\{|\lambda|^\frac{2\xi}{\xi+1} \|F\|_\mathcal{H}\|U\|_\mathcal{H}+\|F||^2_\mathcal{H}\}\quad\rm{for}\quad 0\leq\xi\leq 1.
\end{multline}
Finally,   using    \eqref{Eq119AnalyRRB},    \eqref{Eq119AnalyRRWB} and \eqref{Eq119AnalyRRTB} in the inequality item $(i)$ of Lemma \ref{RegularidadB},   result in 
\begin{multline*}
\beta\gamma[\kappa\|\phi_x+\psi\|^2+b\|A^\frac{1}{2}\psi\|^2]
 \leq
C\big\{ |\lambda|^\frac{-2\tau}{1+\tau}+|\lambda|^\frac{-2\sigma}{1+\sigma}+|\lambda|^\frac{-2\xi}{1+\xi}+|\lambda|^{-1}\big\} \|F\|_\mathcal{H}\|U\|_\mathcal{H}\\
+C\{ |\lambda|^{-\frac{4\tau}{\tau+1}}+ |\lambda|^{-\frac{4\sigma}{\sigma+1}}+ |\lambda|^{-\frac{4\xi}{\xi+1}}\}\|F\|^2_\mathcal{H},
\end{multline*}
then
\begin{multline}\label{Eq020GevreyB}
\beta\gamma|\lambda|[\kappa\|\phi_x+\psi\|^2+b\|A^\frac{1}{2}\psi\|^2]\\
 \leq
C_\delta\left\{ \begin{array}{ccc}
 |\lambda|^\frac{1-3\tau}{1+\tau}\{ |\lambda|^\frac{2\tau}{\tau+1}\|F\|_\mathcal{H}\|U\|_\mathcal{H}+\|F\|^2_\mathcal{H}\} &\text{for} & \tau\leq \sigma\quad{\rm and}\quad \tau\leq \xi, 
 \\\\
  |\lambda|^\frac{1-3\sigma}{1+\sigma}\{ |\lambda|^\frac{2\sigma}{\sigma+1}\|F\|_\mathcal{H}\|U\|_\mathcal{H}+\|F\|^2_\mathcal{H}\} &\text{for} & \sigma \leq \tau\quad{\rm and}\quad \sigma\leq \xi,\\\\
   |\lambda|^\frac{1-3\xi}{1+\xi}\{ |\lambda|^\frac{2\xi}{\xi+1}\|F\|_\mathcal{H}\|U\|_\mathcal{H}+\|F\|^2_\mathcal{H}\} &\text{for} & \xi\leq \sigma\quad{\rm and}\quad \xi\leq \tau.
\end{array}\right.
\end{multline}
Analogously, using  \eqref{Eq119AnalyRRB} and \eqref{Eq119AnalyRRTB} in the inequality item $(ii)$ of Lemma \ref{RegularidadB},   result in 
\begin{equation}\label{Eq021GevreyB}
\delta|\lambda| \|A^\frac{1}{2}\theta\|^2
 \leq
C_\delta \left\{ \begin{array}{ccc}
 |\lambda|^\frac{1-3\tau}{1+\tau}\{ |\lambda|^\frac{2\tau}{\tau+1}\|F\|_\mathcal{H}\|U\|_\mathcal{H}+\|F\|^2_\mathcal{H}\} &\text{for} & \tau\leq \xi, 
\\\\
   |\lambda|^\frac{1-3\xi}{1+\xi}\{ |\lambda|^\frac{2\xi}{\xi+1}\|F\|_\mathcal{H}\|U\|_\mathcal{H}+\|F\|^2_\mathcal{H}\} &\text{for} &  \xi\leq \tau.
\end{array}\right.
\end{equation}
Finally,  summing the    the estimates  \eqref{Eq119AnalyRRB},  \eqref{Eq119AnalyRRWB}, \eqref{Eq119AnalyRRTB}, \eqref{Eq020GevreyB} and  \eqref{Eq021GevreyB} and applying Young inequality, we obtain
\begin{equation*}
\left\{  \begin{array}{ccc}
 |\lambda|^\frac{2\tau}{\tau+1}\|U\|_\mathcal{H}\leq C_\delta|\|F\|_\mathcal{H} &\text{for} & \tau\leq \sigma\quad{\rm and}\quad \tau\leq \xi, 
 \\\\
  |\lambda|^\frac{2\sigma}{1+\sigma}\|U\|_\mathcal{H}\leq C_\delta\|F\|_\mathcal{H} &\text{for} & \sigma \leq \tau\quad{\rm and}\quad \sigma\leq \xi,\\\\
   |\lambda|^\frac{2\xi}{1+\xi}\|U\|_\mathcal{H}\leq C_\delta \|F\|_\mathcal{H} & \text{for} & \xi\leq \sigma\quad{\rm and}\quad \xi\leq \tau.
\end{array}\right.
\end{equation*}
 for $(\tau,\sigma, \xi)\in (0,1)^3$  the proof of this Theorem \ref{GevreyLaminadoB},  is finished.
\end{proof}
\subsubsection{Analyticity of $S(t)=e^{\mathcal{B}_2t}$   for $(\tau,\sigma,\xi)\in \big[\frac{1}{2},  1\big]^3$ such that  $\tau=\xi$.   see Fig. 02}
Before proving the main result of this section,  the following lemma will be proved.
\begin{lemma}\label{Lema001AnaliticityB}
Let $\delta> 0$. There exists a constant $C_\delta > 0$ such that the solutions of \eqref{Eq01B2}--\eqref{Eq09A}
for $|\lambda|\geq  \delta$  satisfy the inequality
\begin{eqnarray}\label{Eq004Lema02B}
(i)\quad |\lambda|[\|\Phi\|^2 +\|\Theta\|^2]\leq  C_\delta\|F\|_\mathcal{H}\|U\|_\mathcal{H}\qquad{\rm for}\qquad \frac{1}{2}\leq\tau=\xi\leq 1.
\\
\label{Eq005Lema02B}
(ii)\quad |\lambda|\|\Psi\|^2\leq  C_\delta\|F\|_\mathcal{H}\|U\|_\mathcal{H}\qquad{\rm for}\qquad \frac{1}{2}\leq\sigma\leq 1.
\end{eqnarray}
\end{lemma}
\begin{proof}
{\bf Item $(i)$:} 
Realizing the duality product of \eqref{esp-20B}  with $\dfrac{\rho_1}{\gamma}A^{-\tau}\lambda \Phi$  and using the property that the operator $A^\nu$ is self-adjoint for all $\nu\in \mathbb{R}$ and using the equation \eqref{esp-10B},  result in 
\begin{eqnarray}
\dfrac{\mu_1}{\gamma}\lambda\|\Phi\|^2 &=& i\dfrac{\kappa}{\gamma}\|A^\frac{1-\tau}{2} \Phi\|^2+i\dfrac{\kappa}{\gamma}\dual{A^\frac{1}{2}f^1}{A^{\frac{1}{2}-\tau}\Phi}+i\dfrac{\kappa}{\gamma}\dual{\Psi}{A^{-\tau}\Phi_x}-i\dfrac{\kappa}{\mu_1}\dual{f_x^3}{A^{-\tau}\Phi}\\
\label{Eq006Lema02B2}
& &-i\dual{f^2}{\Phi} -i\dfrac{\kappa}{\gamma}\dual{f^2}{A^{1-\tau}\phi} +i\dfrac{\kappa}{\gamma}\dual{A^{-\tau}f^2}{\psi_x}+i\dfrac{\rho_1}{\gamma}\|A^{-\frac{\tau}{2}}f^2\|^2  \\
\nonumber
& & -i\dfrac{\rho_1}{\gamma}\lambda^2\|A^{-\frac{\theta}{2}}\Phi\|^2 +\dfrac{\mu}{\gamma}\dual{\Theta}{\lambda A^{-\tau}\Phi_x}.
\end{eqnarray}
On the other hand,  now performing the duality product of \eqref{esp-60B} for $\dfrac{\rho_3}{\gamma}A^{-\xi}\lambda \Theta$, using \eqref{esp-30B}, and as
the operators $A^\nu$ for all $\nu\in\mathbb{R}$ are self-adjoint,  result in 
\begin{eqnarray}
\nonumber
\lambda\|\Theta\|^2 \hspace*{-0.3cm}& = &\hspace*{-0.3cm} -i\lambda^2\dfrac{\rho_3}{\gamma}\|A^{-\frac{\xi}{2}}\Theta\|^2+i\dfrac{\delta}{\gamma}\|A^\frac{1-\xi}{2}\Theta\|^2+i\dfrac{\delta}{\gamma}\dual{A^\frac{1}{2}f^5}{A^{\frac{1}{2}-\xi}\Theta}-i\dfrac{\delta}{\gamma}\dual{f^6}{A^{1-\xi}\theta}-i\dfrac{\mu}{\gamma}\dual{f^6}{A^{-\xi}\Phi_x}\\
\label{Eq006Lema02B3}
& & -i\dfrac{\mu}{\gamma}\dual{f^6}{A^{-\xi}\Psi} -i\dual{f^6}{\Theta}+i\dfrac{\rho_3}{\gamma}\|A^{-\frac{\xi}{2}}f^6\|^2-\dfrac{\mu}{\gamma}\dual{\lambda A^{-\xi}\Phi_x}{\Theta}.
\end{eqnarray}
Imposing the condition $\tau=\xi$ and since ${\rm Re}\{\dual{\Theta}{\lambda A^{-\tau}\Phi_x}-\dual{\lambda A^{-\tau}\Phi_x}{\Theta}\}=0$,  adding the equations \eqref{Eq006Lema02B2} and \eqref{Eq006Lema02B3} and then taking the real part, result in 
\begin{eqnarray*}
\lambda\bigg [ \dfrac{\mu_1}{\gamma}\|\Phi\|^2+\|\Theta\|^2\bigg] &= & i\dfrac{\kappa}{\gamma}\dual{A^\frac{1}{2}f^1}{A^{\frac{1}{2}-\tau}\Phi}+i\dfrac{\kappa}{\gamma}\dual{\Psi}{A^{-\tau}\Phi_x}-i\dfrac{\kappa}{\mu_1}\dual{f_x^3}{A^{-\tau}\Phi}-i\dual{f^2}{\Phi}\\
& & -i\dfrac{\kappa}{\gamma}\dual{f^2}{A^{1-\tau}\phi} +i\dfrac{\kappa}{\gamma}\dual{A^{-\tau}f^2}{\psi_x}  +i\dfrac{\delta}{\gamma}\dual{A^\frac{1}{2}f^5}{A^{\frac{1}{2}-\xi}\Theta}
\\
& & -i\dfrac{\delta}{\gamma}\dual{f^6}{A^{1-\xi}\theta}-i\dfrac{\mu}{\gamma}\dual{f^6}{A^{-\xi}\Phi_x} -i\dfrac{\mu}{\gamma}\dual{f^6}{A^{-\xi}\Psi} -i\dual{f^6}{\Theta}
\end{eqnarray*}
From $|i\frac{\kappa}{\gamma}\dual{\Psi}{A^{-\tau}\Psi_x}|\leq C\{\|\Psi\|^2+\|A^\frac{1-2\tau}{2}\Phi\|^2\}$ and  $\frac{1}{2}\leq\tau=\xi\leq 1$,   applying Cauchy-Schwarz and Young inequalities,  norms $\|F\|_\mathcal{H}$ and $\|U\|_\mathcal{H}$,  the proof of item $(i)$ of this lemma is finished.\\
{\bf Item $(ii)$:}  Similarly,   performing the duality product of \eqref{esp-40B} for $\dfrac{\rho_2}{\mu_2}A^{-\sigma}\lambda \Psi$, using \eqref{esp-30B}, and recalling the                                        self-adjointness of $A^\nu$,   $\nu \in\mathbb{R}$,  lead to 
\begin{eqnarray*}
\lambda\|\Psi\|^2 
\hspace*{-0.2cm}&=&\hspace*{-0.2cm}  -i\dfrac{\rho_2}{\mu_2}\lambda^2\|A^{-\frac{\sigma}{2}}\Psi\|^2+i\dfrac{b}{\mu_2}\|A^\frac{1-\sigma}{2}\Psi\|^2+i\dfrac{b}{\mu_2}\dual{A^\frac{1}{2}f^3}{A^{\frac{1}{2}-\sigma}\Psi}-i\dfrac{\kappa}{\mu_2}\dual{\Phi}{A^{-\sigma}\Psi_x}\\
& & +i\dfrac{\kappa}{\mu_2}\dual{f^1_x}{A^{-\sigma}\Psi}+i\dfrac{\kappa}{\mu_2}\|A^{-\frac{\sigma}{2}}\Psi\|^2+i\dfrac{\kappa}{\mu_2}\dual{f^3}{A^{-\sigma}\Psi} \\
& &+\dfrac{\mu}{\mu_2}\dual{\sqrt{|\lambda|}\Theta}{\dfrac{\lambda}{\sqrt{|\lambda|}}A^{-\sigma}\Psi}-i\dfrac{b}{\mu_2}\dual{f^4}{A^{1-\sigma}\Psi}-i\dfrac{\kappa}{\mu_2}\dual{f^4}{A^{-\sigma}\phi_x}\\
& &-i\dfrac{\kappa}{\mu_2}\dual{f^4}{A^{-\sigma}\psi} +i\dfrac{\mu}{\mu_2}\dual{f^4}{A^{-\sigma}\Theta}-i\dual{f^4}{\Psi}+i\dfrac{\rho_2}{\mu_2}\|f^4\|^2.
\end{eqnarray*}
Noting that:  For $\frac{1}{2}\leq\sigma\leq 1$ lead to   $\frac{1-\sigma}{2}\leq \frac{\sigma}{2}$,  $\frac{1}{2}-\sigma\leq 0$  and $\frac{1-2\sigma}{2}\leq 0$,  on the other hand  as $\|A^{-\sigma} \Psi_x\|^2=\|A^\frac{1-2\sigma}{2}\Psi\|^2$,  taking real part  and considering that  $\frac{1}{2}\leq\sigma\leq 1$  and using Cauchy-Schwarz and Young  inequalities,  norms $\|F\|_\mathcal{H}$ and $\|U\|_\mathcal{H}$,   for $\varepsilon>0$,  exists $C_\varepsilon>0$  independent of $\lambda$ such that
\begin{eqnarray*}
|\lambda|\|\Psi\|^2 & \leq & C_\delta\|F\|_\mathcal{H}\|U\|_\mathcal{H}+C\{\|\Phi\|^2+\|\Psi\|^2\}+C_\varepsilon|\lambda|\|\Theta\|^2+\varepsilon|\lambda|\|\Psi\|^2.
\end{eqnarray*}
From   \eqref{EstimaEquivExp2} and item $(i)$ of this lemma,    the proof of item $(ii)$ of this lemma is finished.
\end{proof}
\begin{theorem}
The semigroup $S(t)=e^{\mathcal{B}_2t}$ associated to the system  \eqref{Eq01B2}--\eqref{Eq09A} is analytic when the three parameters $\tau$, $\sigma$ and $\xi$ vary in the interval $[\frac{1}{2},1]$ with $\tau=\xi$:
\begin{center}
\tdplotsetmaincoords{80}{-35}
\begin{tikzpicture}[tdplot_main_coords, scale=4.5,]
    \coordinate(A) at (0.5,0.5,0.5);
    \coordinate(B) at (1,0.5,0.5);
    \coordinate(C) at (1,0.5,1);
    \coordinate(D) at (0.5,0.5,1);
    \coordinate(E) at (0.5,1,0.5);
    \coordinate(F) at (1,1,0.5);
    \coordinate(G) at (1,1,1);
    \coordinate(H) at (0.5,1,1);
    \filldraw[fill=blue!20](A)--(C)--(G)--(E);
    \draw (A)--(B)--(C)--(D)--(A);
    \draw [dashed] (E)--(F);
    \draw (E)--(H)--(G);
    \draw [dashed] (F)--(G);
    \draw (A)--(E);
    \draw [dashed] (B)--(F);
    \draw (C)--(G);
    \draw (D)--(H);
    \draw [blue, dashed] (A)--(C);
    \draw [blue, dashed] (E)--(G);
    \draw[->, black!60] (0, 0,0.3) -- (1.1, 0,0.3);
    \draw[->, black!60] (0, 0,0.3) -- (0, 1.1,0.3);
    \draw[->, black!60] (0, 0,0.3) -- (0, 0,1.6);
    \draw node at (1.15, 0,0.3) {\Large $\tau$};
    \draw node at (0, 1.15,0.3) {\Large $\sigma$};
    \draw node at (0, 0,1.65) {\Large $\xi$};
    \draw node at (1,1,1.04) { \tiny  $(1,1,1)$};
    \draw node at (1.14,0.5,1) {\tiny $\left(1,\dfrac{1}{2},1\right)$};
    \draw node at (0.27,0.85,0.545) {\tiny  $\left(\dfrac{1}{2},1,\dfrac{1}{2}\right)$};
    \draw node at (0.52, 0.5, 0.44) {\tiny $\left(\dfrac{1}{2},\dfrac{1}{2},\dfrac{1}{2}\right)$};
    \draw[fill=black](1,1,1) circle (0.35pt);
    \draw[fill=black](0.5,0.5,0.5) circle (0.35pt);
    \draw[fill=black](C) circle (0.35pt);
    \draw[fill=black](E) circle (0.35pt);
    
\end{tikzpicture}
\end{center}
\begin{center}
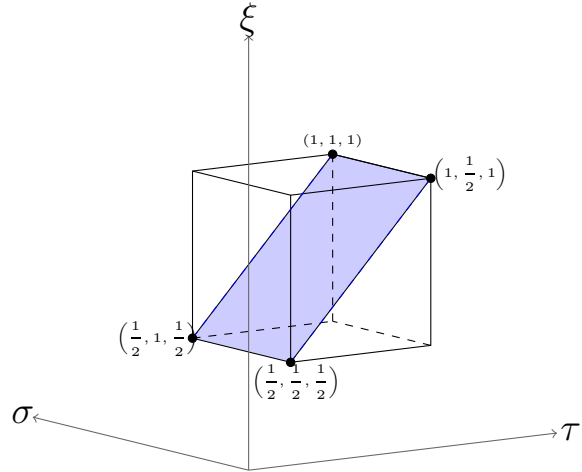
{\bf FIG. 02:} Region $R_{A2}$ of Analyticity de $S(t)=e^{\mathcal{B}_2t}$
\end{center}
\end{theorem}
\begin{proof}
This theorem will be proved again using now the Theorem \ref{LiuZAnalyticity},  so the two conditions \eqref{EixoImaginary} and \eqref{Analiticity} must be proved.

Here the condition's test  \eqref{EixoImaginary} will be omitted because it is completely similar to the test already conducted for the first  the first system.
 
Next, the condition \eqref{Analiticity} is proved,  note that proving this condition is equivalent to show,  let $\delta>0$. There exists a constant $C_\delta > 0$ such that the solutions of \eqref{Eq01B2}--\eqref{Eq03B2} and \eqref{Eq08A}--\eqref{Eq09A}
for $|\lambda|\geq  \delta$  satisfy the inequality
 \begin{equation}\label{EquivAnaliticityB}
 |\lambda|\|U\|^2_\mathcal{H}\leq C_\delta\|F\|_\mathcal{H}\|U\|_\mathcal{H}.
 \end{equation}
 It is not difficult to see that this inequality \eqref{EquivAnaliticity} follows from the inequalities of the Lemmas \ref{RegularidadB} and \ref{Lema001AnaliticityB},    so the proof of this theorem is finished.
\end{proof}
\begin{remark}[Asymptotic Behavior]\label{OBS}
It is emphasized that thanks to the  Lemmas \ref{LemaExponencial} and \ref{EixoImaginary01}  the exponential decay of the semigroup  $S_1(t)=e^{\mathcal{B}_1t}$ is obtained when the 3 parameters take values in the closed interval $[0,1]$, also for the second system using the  Lemmas \ref{LemaExponencial2}  and proving in a completely similar way to the proof of Lemma \ref{EixoImaginary01}  which  $i\mathbb{R}\subset\rho(\mathcal{B}_2)$ the corresponding semigroup $S_2(t)=e^{\mathcal{B}_2t}$,  will be exponentially stable when the 3 parameters take values in the closed interval $[0, 1 ] $.
\end{remark}


\end{document}